\documentclass[3p,times]{elsarticle}
\pdfoutput=1

\usepackage{amsmath}
\usepackage{amssymb}
\usepackage{amsfonts}
\usepackage{amsthm}

\usepackage{geometry}
\usepackage{indentfirst}

\usepackage{tabularx}
\usepackage{booktabs}
\usepackage{graphicx}

\usepackage{graphicx}

\usepackage{ulem}

\usepackage{pgfplots} 
\pgfplotsset{compat=newest}
\usetikzlibrary{fpu,fixedpointarithmetic,babel,external,arrows.meta}
\usepackage{fp}

\usepackage{xspace}

\usepackage{xcolor}

\usepackage{aliascnt}

\newcommand{\lp}{\left (}
\newcommand{\rp}{\right )}
\newcommand{\bp}{\begin{pmatrix}}
\newcommand{\ep}{\end{pmatrix}}

\newcommand{\R}{\mathbb{R}}

\newcommand{\pt}{\partial_t}
\newcommand{\px}{\partial_x}

\newcommand{\dt}{{\Delta t}}
\newcommand{\dx}{{\Delta x}}

\newcommand{\eps}{\varepsilon}
\newcommand{\F}{\mathcal{F}}
\newcommand{\jmh}{{j - \frac 1 2}}
\newcommand{\jph}{{j + \frac 1 2}}
\newcommand{\minmod}{\textsf{minmod}\xspace}

\newcommand{\D}{\mathcal{D}}

\newcommand{\tap}{TVD-AP\xspace}       
\newcommand{\tapm}{AP-MOOD\xspace} 

\newtheorem{algorithm}{Algorithm}
\newtheorem{lemma}{Lemma}
\newtheorem{remark}{Remark}
\newtheorem{theorem}{Theorem}
\newtheorem{proposition}{Proposition}

\begin{document}

\begin{frontmatter}
\title{Second order Implicit-Explicit Total Variation Diminishing schemes for the Euler system in the low Mach regime}
\author[ferrara]{Giacomo Dimarco}
\ead{giacomo.dimarco@unife.it}

\author[imb]{Rapha{\"e}l Loub{\`e}re}
\ead{raphael.loubere@math.u-bordeaux.fr}

\author[imt]{Victor Michel-Dansac}
\ead{Victor.Michel-Dansac@math.univ-toulouse.fr}

\author[imt]{Marie-H{\'e}l{\`e}ne Vignal$\*$}
\ead{mhvignal@math.univ-toulouse.fr}
\cortext[cor1]{Corresponding author}

\address[ferrara]{Department of Mathematics and Computer Science, University of
Ferrara, Ferrara, Italy}
\address[imb]{CNRS and Institut de Math\'{e}matiques de Bordeaux (IMB)
 Universit{\'e} de Bordeaux, France}
\address[imt]{Institut de Math\'{e}matiques de Toulouse (IMT),
Universit{\'e} P. Sabatier, Toulouse}

\begin{abstract}
In this work, we consider the development of implicit explicit total variation diminishing (TVD) methods (also termed SSP: strong stability preserving) for the compressible isentropic Euler system in the low Mach number regime. The scheme proposed is asymptotically stable with a CFL condition independent from the Mach number and it degenerates in the low Mach number regime to a consistent discretization of the  incompressible system. Since, it has been proved that implicit schemes of order higher than one cannot be TVD (SSP) \cite{GotShuTad}, we construct a new paradigm of implicit time integrators by coupling first order in time schemes with second order ones in the same spirit as highly accurate shock capturing TVD methods in space. For this particular class of schemes, the TVD property is first proved on a linear model advection equation and then extended to the isentropic Euler case. The result is a method which interpolates from the first to the second order both in space and time, which preserves the monotonicity of the solution, highly accurate for all choices of the Mach number and with a time step
only restricted
by the non stiff part of the system. In the last part, we show thanks to one and two dimensional test cases that the method indeed possesses the claimed properties.
\end{abstract}

\end{frontmatter}

\section{Introduction} \label{sec:intro}
The analysis \cite{KlaMaj,KlaMaj2,Sch2,Asano, LioMas,Ala} and the development of numerical methods \cite{Issa,Harlow,Patankar,Turkel,Klein,Collela,Guillard,Viozat,Munz,Heul,   GH,Munz2,Munz3,Deg_Del_Sang_Vig,Dellacherie1,DegTan,Degond2,Haack,Chalon,Vila,Kheriji,Dellacherie2,Chalon2,DimLouVig} for the passage from compressible to incompressible gas dynamics has been and is still a very active field of research. The compressible Euler equations which describe conservation of density, momentum and energy in a fluid flow become stiff when the Mach number tends to zero. This implies a fluid flow almost at rest. In this case, the pressure waves move at very large speed compared to the average speed of the gas. Thus, a standard model approximation consists in replacing the density conservation equation by a constraint on the velocity divergence, set consequently equal to zero. In addition, the momentum equation could be replaced by an elliptic equation for the pressure. We refer to that situation to as the incompressible Euler model which is used to describe many different flow conditions. However, there are situations in which the Mach number may be small in some part of the domain and large in others or may strongly change in time. In these cases, one should deal with the coupling of incompressible and compressible regions the topologies of which change in time. This causes from the numerical point of view many difficulties since standard domain decomposition techniques which couple the solution of the compressible equations with the solution of the incompressible system may be difficult to use \cite{Bog}. Thus, one solution consists in solving the more complete compressible Euler system also in the stiff regime. However, this introduces strong drawbacks since the Mach number may become extremely small causing severe time step limitations. To circumvent these problems, in the recent past, asymptotic stable techniques have been
developed~\cite{Degond3,Deg_Del_Sang_Vig,DegTan,Degond2,Haack,Tang,Chalon,Noelle,Chalon2,DimLouVig}. These techniques
permit to compute the solution of such stiff problems avoiding time step limitations directly related to the low Mach number regime. In addition, these methods
lead to
consistent approximations of the
limit incompressible system when the Mach number goes to zero. In this context, in a recent work \cite{DimLouVig}, a first order asymptotic preserving method has been developed. In particular, this work dealt with an analysis of the stability properties which
led to
a stability restriction on the numerical method independent from the Mach number.
The $L^2$ and $L^\infty$ properties of the method have been analyzed in detail.

In the present work, we extend the previous study to the second order in time and space situation.
We first present a second order extension of our previous method which is $L^2$ stable. Successively, since it has been proved~\cite{GotShuTad} that the $L^\infty$ and Total Variation Diminishing (TVD, also named Strong Stability Preserving, SSP) properties cannot be assured for an unconstrained implicit in time scheme of order greater than one, we construct a new paradigm of highly accurate implicit in time schemes. We stress that as opposed to several recent studies on IMEX-SSP methods or Implicit-SSP methods \cite{Ferracina,Hig2,Ket,Constantinescu,Hig,Song,Got} in which the authors look for the largest possible time step which allows the SSP property to be preserved, here we do not pursue in this direction since the stiffness of the equations typically requires numerical methods in which the time steps are disconnected from the stiff scales, and, possibly, several orders of magnitude larger. This as shown in \cite{GotShuTad} will not be possible with standard IMEX-SSP methods. For this reason, the direction chosen in this work consists in constructing a TVD Asymptotic Preserving (AP) scheme using a convex combination of first and second order implicit-explicit (IMEX) methods in the same spirit as high resolution shock capturing TVD methods in space~\cite{LeV2,Toro}. This permits to  prove that our TVD AP scheme possesses both $L^\infty$ and TVD  properties and it opens the way to the construction of arbitrarily high order accurate methods with the same properties. Details of this approach and development of high resolution schemes by combining schemes of order higher than two with first order implicit methods in the case of the linear and non linear transport equations are currently under study \cite{DLV}.

In a second part of the work, we discuss limiters which allow to detect the
troubled
situations in which the TVD property is violated and
 subsequently
 to pass from the second order accurate scheme to the TVD-AP scheme without losing accuracy. The approach proposed is based on the so-called MOOD (Multidimensional Optimal Order Detection) method~\cite{ClaDioLou,DioClaLou,DioLouCla} originally developed to detect the loss of physical properties in space of high resolution methods and to reduce the order of the space discretization to restore the physical properties of the problem. Here, we extend the previous method to the case of the implicit time discretizations. Thus resuming, the proposed method ensures a non oscillatory approximation of our original problem which is more accurate than the one given by a first-order AP scheme, stable independently on the Mach number and which degenerates to an high order time-space discretization of the incompressible Euler equations in the limit
when the Mach number goes to zero.

The article is organized as follows. In Section~\ref{presentation}, we briefly recall the isentropic/ isothermal system of Euler equations and its low Mach number limit as well as the first  order accurate Asymptotic preserving scheme presented in~\cite{DimLouVig} which is the basis of the second order extension here considered. Then, in
Section~\ref{sec:asymptotically_accurate_scheme}, we present a second order AP scheme in time for the isentropic Euler system and we show that even if stable, it presents some non physical oscillations when the explicit C.F.L. condition is violated. Thus, in Section~\ref{model_pb}, we introduce a model problem which will be used to construct a TVD AP scheme for the Euler system and we study its TVD, $L^\infty$ and $L^2$ stability properties. In Section \ref{APEuler} we extend the previous scheme to the low Mach number case and we introduce the MOOD procedure to detect the loss of $L^\infty$ stability.
Finally, in Section~\ref{numeric}, we show the good behavior of our AP schemes with different numerical results for three different one dimensional test cases and a two-dimensional
one. A concluding Section ends the paper.

\section{The  first-order asymptotic-preserving scheme for the Euler system in the low Mach number limit}\label{presentation}

We consider a bounded polygonal domain $\Omega \in \R^d$, where $d \in \{ 1, 2, 3 \}$.
The space and time variables are respectively denoted by $x \in \Omega$ and $t \in \R_+$.
We study the isentropic/isothermal rescaled Euler model with  $\varepsilon> 0$ the squared Mach number (see \cite{KlaMaj,MetSch} for instance).
This reads
\begin{equation}
    \label{eq:isentropic_2D}
    \begin{array}{l}
        \displaystyle\pt \rho + \nabla \cdot ( \rho U ) = 0, \\
        \displaystyle\pt ( \rho U ) + \nabla \cdot \lp \rho U \otimes U \rp + \frac 1 \varepsilon \nabla p(\rho) = 0,
    \end{array}
\end{equation}
where $\rho(t,x) > 0$ is the density of the fluid,
$U(t,x) \in \R^d$ its velocity,
and $p(\rho) = \rho^\gamma$ its pressure.
The parameter $\gamma \geq 1$ is the ratio of specific heats, $\gamma = 1$ corresponds to isothermal fluids while~$\gamma > 1$ to isentropic ones.

Equipped with suitable initial and boundary conditions (consistent with the limit model),
system (\ref{eq:isentropic_2D}) tends to the incompressible isentropic/isothermal Euler system when $\varepsilon\rightarrow 0$
(see \cite{KlaMaj,KlaMaj2,Sch2,Asano,LioMas,MetSch,Ala} for rigorous results).
A formal derivation of this limit is presented for instance in \cite{DimLouVig}.
In the above cited works, the authors introduce well-preparedness and incompressibility assumptions on the initial and boundary conditions to show that
the equations (\ref{eq:isentropic_2D}) tend to the following incompressible Euler equations in the low Mach number limit $\varepsilon \to 0$:
\begin{equation}
    \label{eq:incompressible_Euler_2D}
    \begin{array}{l}
\displaystyle        \rho = \rho_0, \\
\displaystyle        \nabla \cdot U = 0, \\
 \displaystyle       \rho_0 \pt U + \rho_0 \nabla \cdot \lp U \otimes U \rp + \nabla \pi_1 = 0,
    \end{array}
\end{equation}
where the first-order correction of the pressure, denoted by $\pi_1$, is implicitly defined by the incompressibility constraint $\nabla \cdot U = 0$.

We now briefly recall the numerical method \cite{DimLouVig} which serves as a basis for the new method introduced in the next Section.
The discretization of the space and time domains follows the usual finite volume framework. The solution $W(t,x)=(\rho,\rho U)(t,x)$ of the Euler equations is
approximated at time $t^n = n \dt$, where $\dt$ is the time step, by $W^n$.
The scheme relies on an IMEX (IMplicit-EXplicit) decomposition of system (\ref{eq:isentropic_2D}) (see \cite{Deg_Del_Sang_Vig,DegTan,Tang,DimLouVig}):
for all $n \geq 0$, the semi-discrete in time form of the scheme reads
\begin{equation}
    \label{eq:first_order_scheme_semi_discrete_coupled}
    \frac { W^{n+1} - W^n } \dt + \nabla \cdot F_e(W^n) + \nabla \cdot F_i(W^{n+1}) = 0.
\end{equation}
where $    F_e(W^n) = (0 ,\rho^n U^n \otimes U^n )$  and
   $ F_i(W^{n+1}) =(\rho^{n+1} U^{n+1}, p(\rho^{n+1}) / \varepsilon \ \mathbb{I}_2)$ and $F_e$ is taken explicitly while $F_i$ implicitly.
Note that the two systems associated to the two fluxes taken singularly are hyperbolic. An interesting property of such approach is that the resolution of the two equations
composing the system can be decoupled. Indeed, in (\ref{eq:first_order_scheme_semi_discrete_coupled}), taking the divergence of the momentum equation and
inserting it into the mass equation yields:
\begin{subequations}
    \label{eq:first_order_scheme_semi_discrete}
    \begin{align}
        \label{eq:first_order_density}
        \frac { \rho^{n+1} - \rho^n } \dt &+ \lp \nabla \cdot \lp \rho U \rp \rp^n
                                           - \dt \lp \nabla^2 : \lp \rho U \otimes U \rp \rp^n
                                           - \frac \dt \varepsilon \lp \Delta p(\rho) \rp^{n+1} = 0, \\
        \label{eq:first_order_momentum}
        \frac { (\rho U)^{n+1} - (\rho U)^n } \dt &+ \lp \nabla \cdot \lp \rho U \otimes U \rp \rp^n + \frac 1 \varepsilon \lp \nabla p(\rho) \rp^{n+1} = 0,
    \end{align}
\end{subequations}
where $\nabla^2$ and $:$ are respectively the tensor of second order derivatives and the contracted product of two tensors.
Then, one can solve first the nonlinear equation (\ref{eq:first_order_density}) which gives $\rho^{n+1}$, and, successively get
the momentum
from (\ref{eq:first_order_momentum}). The implicit treatments of the pressure gradient and the mass flux respectively provide the asymptotic consistency and the uniform stability of the scheme \cite{DimLouVig}.

We now present the space discretization, in one space dimension for the sake of clarity.
The space domain is assumed to be partitioned in cells of center $x_j$ and size $\dx$. Then,
on~$[t^n, t^{n+1})$ the fully discrete version of (\ref{eq:first_order_scheme_semi_discrete}) reads as follows
\begin{equation}
    \label{eq:first_order_scheme_fully_discrete}
    \begin{aligned}
        \frac { W_j^{n+1} - W_j^n } \dt + \frac { (\F_e)_{\jph}^n - (\F_e)_{\jmh}^n } \dx + \frac { (\F_i)_{\jph}^{n,n+1} - (\F_i)_{\jmh}^{n,n+1} } \dx
                                         - \dt \bp \lp \Delta(\rho u^2) \rp_j^n +\displaystyle \frac{1}{\varepsilon}\, \lp \Delta p(\rho) \rp_j^{n+1}\cr 0 \ep = 0.
    \end{aligned}
\end{equation}
The explicit numerical flux is given by
\begin{equation}
    \label{eq:explicit_flux}
    (\F_e)_{\jph}^n := \frac {F_e(W_j^n) + F_e(W_{j+1}^n)} 2 + (\D_e)_{\jph}^n ( W_{j+1}^n - W_j^n ),
\end{equation}
with $(\D_e)_{\jph}^n$ the explicit viscosity coefficient, taken as half of the maximum explicit eigenvalue and given by
$(\D_e)_{\jph}^n :=  \max \lp |u_j^n|,|u_{j+1}^n| \rp$. The implicit numerical flux is given by
\begin{equation}
    \label{eq:implicit_flux}
   (\F_i)_{\jph}^{n,n+1} :=  \frac {F_i(W_j^{n,n+1}) + F_i(W_{j+1}^{n,n+1})} 2 + (\D_i)_{\jph}^n ( W_{j+1}^{n+1} - W_j^{n+1} ),
\end{equation}
where $W_j^{n,n+1}=(\rho^{n+1}_j, q_j^n)$ and
$(\D_i)_{\jph}^n$ is the implicit viscosity coefficient, taken as half of the maximum implicit eigenvalue
$ (\D_i)_{\jph}^n := \frac 1 2 \max \lp \sqrt{ p'(\rho_j^n)/ \varepsilon }, \sqrt{ p'(\rho_{j+1}^n)/ \varepsilon } \rp$. This choice for the implicit viscosity is enough to get an $L^\infty$ stable scheme. However, by relaxing the request and fixing it to zero one can show that an $L^2$ stable scheme is obtained. Finally, the second-order derivatives are approximated by classical second order centered differences while
the time step is constrained by the following uniform C.F.L. condition:
\begin{equation}
    \label{eq:CFL_first_order}
    \dt \leq \frac{\dx}{ \max_j \lp 2 |u_j^n| \rp}.
\end{equation}
Note that $2\, u$ corresponds to the first eigenvalue of the explicit flux and as expected, this C.F.L. condition does not depend on the Mach number~$\varepsilon$.
When $\varepsilon$ tends to $0$, this scheme yields a consistent discretization of the incompressible system (\ref{eq:incompressible_Euler_2D}). In the following Sections we discuss
an extension of this method to the case of high order time and space discretizations.

\section{A second order asymptotically accurate scheme for the isentropic Euler equations}
\label{sec:asymptotically_accurate_scheme}
The second order in time extension of the method described in the previous Section is based on an Implict-Explicit (IMEX) Runge-Kutta approach \cite{AscRuuSpi,ParRus,ParRus2,dimarco1,BosRus,dimarco2,BosRus2,BisLukYel}. In particular, we make use of the second order Ascher, Ruuth and Spiteri \cite{AscRuuSpi} scheme denoted in the sequel by ARS(2,2,2). Let us observe that this scheme has been originally constructed to deal with convection-diffusion equations and in particular to deal with cases in which the diffusion (the fast scale) is taken implicit while the convection explicit. In our case, the problem is different, since the fast and the slow scales are both of hyperbolic type.
This causes a real challenge from the numerical point of view. In fact, as already mentioned, implicit methods of order higher than one for hyperbolic problems cannot be TVD \cite{GotShuTad}. Moreover, the situation does not change when implicit-explicit methods are employed as shown later. Thus, one can think to rely on optimal IMEX-SSP methods \cite{Ket,Got,Hig,Song} to solve the problem. However, in this case, time steps allowing the TVD property to be preserved are of the order of explicit time integrators. Unfortunately, since we are dealing with a limit problem, we look for a method which preserves the TVD property independently on the Mach number which eventually can also be set to zero. Thus, in order to bring remedy to this problematic situation, the idea explored in this work consists in blending together first and second order implicit time-space discretizations giving rise to a new class of high resolution in time methods which guarantees the preservation of the  $L^\infty$ stability and TVD property.
Here, we discuss two given implicit time discretizations, while we refer to \cite{DLV} for the construction of general TVD high resolution implicit-explicit time discretizations.

The Butcher tableau relative to the considered ARS(2,2,2) scheme is detailed in Table~\ref{tab:Butcher_tableaux_ARS222}  with $\beta = 1 - \sqrt{2}/ 2$ and $\alpha=1-1/(2\beta)$. Note that on the left is reported the explicit tableau applied to the flux $F_e$ while on the left the implicit tableau applied to the flux $F_i$.
\begin{table}[!ht]
    \hfill
    \begin{tabular}{c|ccc}
        0       & 0           & 0           & 0 \\
        $\beta$ & $\beta$     & 0           & 0 \\
        1       & $\alpha$  & 1 - $\alpha$ & 0 \\ \hline
                & $\alpha$  & 1 - $\alpha$ & 0
    \end{tabular}
    \hfill
    \begin{tabular}{c|ccc}
        0       & 0 & 0           & 0       \\
        $\beta$ & 0 & $\beta$     & 0       \\
        1       & 0 & 1 - $\beta$ & $\beta$ \\ \hline
                & 0 & 1 - $\beta$ & $\beta$
    \end{tabular}
    \hfill \vphantom{phantom}
    \caption{Butcher tableaux for the ARS(2,2,2) time discretization.
             Left panel: explicit tableau.
             Right panel: implicit tableau.}
    \label{tab:Butcher_tableaux_ARS222}
\end{table}
Remarking that $\alpha=\beta-1$ and so $1-\alpha=2-\beta$,
the corresponding semi-discretization of the Euler system is
given by
\begin{subequations}\label{semi_ARS}
\begin{equation}\label{semi_ARS_step_one}
\frac{W^\star-W^n}{\Delta t}+\beta\, \nabla\cdot F_e(W^n)+\beta\, \nabla\cdot F_i(W^\star)=0,
\end{equation}
\begin{equation}\label{semi_ARS_step_two}
\frac{W^{n+1}-W^n}{\Delta t}+(\beta-1)\, \nabla\cdot F_e(W^n)+(2-\beta)\, \nabla\cdot F_e(W^\star)+(1-\beta)\, \nabla\cdot F_i(W^\star)+\beta\,\nabla\cdot F_i(W^{n+1})=0.
\end{equation}
\end{subequations}
Likewise for the first-order accurate scheme, the previous second-order accurate discretization has an uncoupled formulation.
Let us first establish
the first step~(\ref{semi_ARS_step_one}).
Taking the divergence of the momentum equation of (\ref{semi_ARS_step_one}) and inserting the value of $\nabla\cdot (\rho U)^\star$ into the mass equation of (\ref{semi_ARS_step_one}), yield the following uncoupled formulation:
\begin{equation*}
\frac{W^\star-W^n}{\Delta t}+\beta\, \nabla\cdot F_e(W^n)+\beta\, \nabla\cdot F_i(W^{n,\star})
-\beta^2\, \Delta t\left(\!\!\begin{array}{c}
\nabla^2:(\rho U \otimes U )^n+\displaystyle\frac{1}{\varepsilon}\, \Delta p(\rho^\star)\\0
\end{array}
\!\!\right)
=0,
\end{equation*}
where $W^{n,\star}=(\rho^\star,(\rho U)^n)$.
Using now the same notation as in the previous section for the first-order accurate scheme, the full discretized uncoupled first step in one dimension, is given by
\begin{subequations}\label{semi_ARS_uncoupled}
\begin{equation}
    \label{eq:second_order_fully_discrete_step_one}
        \frac { W_j^{\star} - W_j^n } \dt +\beta\,  \frac { (\F_e)_{\jph}^n - (\F_e)_{\jmh}^n } \dx + \beta\, \frac { (\F_i)_{\jph}^{n,\star} - (\F_i)_{\jmh}^{n,\star} } \dx
                                         - \beta^2\, \dt \bp \lp \Delta(\rho u^2) \rp_j^n +\displaystyle \frac{1}{\varepsilon}\, \lp \Delta p(\rho) \rp_j^{\star}\cr 0 \ep = 0.
\end{equation}
We turn to the uncoupled formulation of the second step~(\ref{semi_ARS_step_two}).
We insert
the divergence of $\rho\, U^{n+1}$, obtained with the momentum equation of~(\ref{semi_ARS_step_two}),
into the mass equation of (\ref{semi_ARS_step_two}).
This yields
$$ \displaylines{\frac{W^{n+1}-W^n}{\Delta t}+(\beta-1)\, \nabla\cdot F_e(W^n)+(2-\beta)\, \nabla\cdot F_e(W^\star) +(1-\beta)\, \nabla\cdot F_i(W^\star)
+\beta\,\nabla\cdot F_i(W^{n,n+1})\hfill\cr
\hfill-\beta\, \Delta t \left(\!\!\begin{array}{c}
(\beta-1)\, \nabla^2:(\rho U \otimes U )^n+(2-\beta)\,\nabla^2:(\rho U \otimes U )^\star+ \displaystyle\frac{(1-\beta)}{\varepsilon}\, \Delta p(\rho^\star)
+\frac{\beta}{\varepsilon}\, \Delta p(\rho^{n+1})
\\0
\end{array}
\!\!\right) = 0.}
$$
Using the same notation as before, the fully discretized second step in one dimension is given by
\begin{equation}
    \label{eq:second_order_fully_discrete_step_two}
         \begin{array}{l}
       \displaystyle\frac{W^{n+1}_j-W^n_j}{\Delta t}+(\beta-1)\,\frac { (\F_e)_{\jph}^n - (\F_e)_{\jmh}^n } \dx+(2-\beta)\, \frac { (\F_e)_{\jph}^\star - (\F_e)_{\jmh}^\star } \dx  \\
        [8pt]\hspace{1.8cm}  \displaystyle               +(1-\beta)\, \frac { (\F_i)_{\jph}^{\star,\star} - (\F_i)_{\jmh}^{\star,\star} } \dx
        +\beta\,\frac { (\F_i)_{\jph}^{n,n+1} - (\F_i)_{\jmh}^{n,n+1} } \dx\\
       [8pt] \hspace{1.8cm}                    \displaystyle             -\beta\, \Delta t \left(\!\!\begin{array}{c}
(\beta-1)\, \,\lp\Delta(\rho u^2) \rp_j^n+(2-\beta)\,\lp\Delta(\rho u^2) \rp_j^\star+ \displaystyle\frac{ (1-\beta)}{\varepsilon}\,  \lp \Delta p(\rho) \rp_j^{\star}
+\frac{\beta}{\varepsilon}\, \lp \Delta p(\rho) \rp_j^{n+1}
\\
[10pt]0
\end{array}
\!\!\right) = 0.
    \end{array}
\end{equation}

\end{subequations}

\begin{lemma}
\!\!The scheme (\ref{semi_ARS}) is asymptotically consistent with system (\ref{eq:incompressible_Euler_2D}) in the limit $\varepsilon \rightarrow 0$.
\end{lemma}
\begin{proof}
We do not assume well-prepared initial conditions but general initial conditions $\rho(0,x)=\rho^0(x)$ and $U(0,x)=U^0(x)$. Well prepared initial conditions will converge to a constant density and a divergence free velocity when $\eps$ tends to $0$. We consider the boundary condition
$U(x,t)\cdot\nu(x)=0$ for all $t\geq 0$ and all $x\in\partial\Omega$ the boundary of $\Omega$,
where $\nu$ is the outward unit normal.

We assume that all discrete quantities (densities and momentums) have a limit when $\eps\rightarrow 0$, then at the first time-step $n=0$, multiplying the momentum equation of (\ref{semi_ARS_step_one}) and letting $\eps$ tends to $0$,
 gives $\nabla p(\rho^{\star}) = 0$ and so  $\nabla \rho^{\star}=0$. Then,
integrating the mass equation of (\ref{semi_ARS_step_one}) on the domain and using the boundary condition $U^\star\cdot \nu=0$ on $\partial\Omega$, one gets
$\rho^{\star}=<\rho^0>=1/|\Omega|\int_\Omega \rho^0(x)\, dx$. Similarly, the second stage (\ref{semi_ARS_step_two}) gives $\nabla \rho^{1}=0$, and integrating the mass equation and using the boundary conditions $U^\star\cdot \nu=U^1\cdot \nu=0$, we obtain $\rho^1=\rho^{\star}=<\rho^0>$.

Note that inserting this result into the mass equation of the first stage (\ref{semi_ARS_step_one}), we do not recover the incompressibility constraint for the first time-step
since $\nabla\cdot U^\star(x)=(<\rho^0>-\rho^0(x))/(\Delta t\, <\rho^0>)$ which equals $0$ if and only if the initial density $\rho^0$ is well prepared and tends to a constant when $\eps$ tends to $0$. But, for all $n\geq 1$, we recover the incompressibility constraint for the first stage $\nabla\cdot U ^{\star}=0$.
Finally, thanks to $\nabla\cdot U ^{\star}=0$, the density equation gives $\nabla\cdot U ^{n+1}=0$ for all $n\geq 1$.
Consequently, the scheme projects the solution over the asymptotic incompressible limit even if the initial data are not well-prepared to this limit,
we obtain $\rho^{n+1}=<\rho^0>:=\rho_0$ and $\nabla\cdot U^{n+1}=0$ for all $n\geq 1$.
Concerning the pressure, for $n\geq 1$, the limit scheme becomes
\begin{subequations}
\begin{equation}
\rho_0\, \frac{U^{\star}-U^n}{\Delta t} +\beta\rho_0 \nabla\cdot(U\otimes U)^n
                    + \beta\nabla \pi_1^{\star}  = 0,
\end{equation}
\begin{equation}
\rho_0\, \frac{U^{n+1}-U^n}{\Delta t} +\alpha\rho_0 \nabla\cdot(U\otimes U)^n +(1-\alpha)\rho_0 \nabla\cdot(U\otimes U)^\star
                   + (1-\beta)\nabla \pi_1^{\star} + \beta\nabla \pi_1^{n+1}  = 0,
\end{equation}
\end{subequations}
where $\pi_1^{\star}=\lim_{\eps\rightarrow 0}\frac{1}{\eps}\left(p(\rho^{\star})-p(\rho_0)\right)$ and $\pi_1^{n+1}=\lim_{\eps\rightarrow 0}\frac{1}{\eps}\left(p(\rho^{n+1})-p(\rho_0)\right).$
\end{proof}

Now, we test  this second order accurate uncoupled AP scheme~(\ref{semi_ARS_uncoupled}) on a shock tube test case (see Section~\ref{Sod} for the details).
The C.F.L is uniform and given by~(\ref{eq:CFL_first_order}).
The exact solution is constituted of a rarefaction wave and a shock wave and Figure \ref{fig:Riemann_problem_second_order_time} reports the numerical solution of
the first and second order AP schemes described above for different values of the Mach number $\varepsilon$ while the space discretization is always first order.
\begin{figure}[!ht]
    \centering
    \hspace{-0.85cm}
    \includegraphics[height=0.23\textwidth]{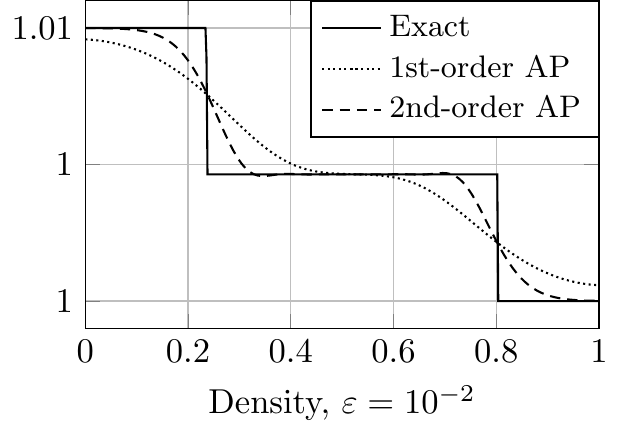} %
    \includegraphics[height=0.23\textwidth]{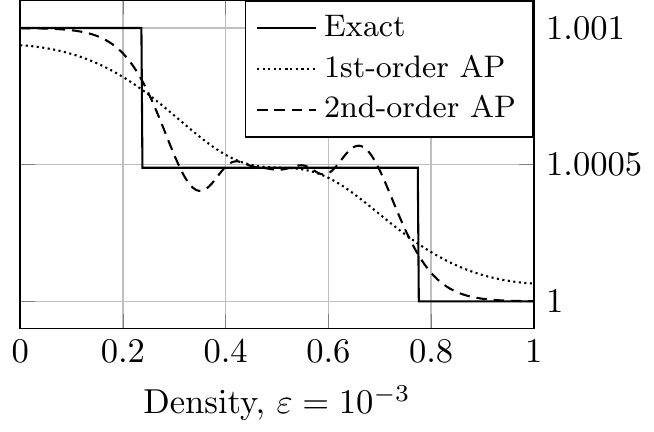}
    \includegraphics[height=0.23\textwidth]{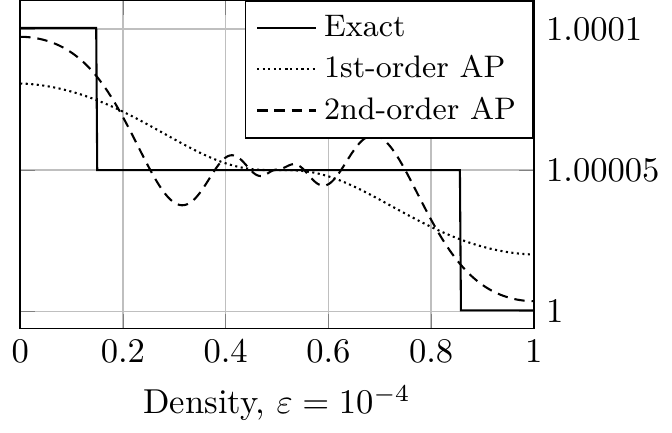}
    \caption{Approximations of the density $\rho$ for a rarefaction-shock Riemann problem and different values of the Mach number.}
    \label{fig:Riemann_problem_second_order_time}
\end{figure}
As we can see, the second-order AP scheme gives more accurate results but it presents some oscillations when the Mach number decreases for a fixed value of the time step given by the C.F.L condition (\ref{eq:CFL_first_order}). These oscillations disappear when the time step is reduced and in particular when the  non uniform explicit C.F.L condition is satisfied
($ \dt \leq \dx/( \max_j \lp  |u_j^n\pm\sqrt{p'(\rho^n_j)/\eps} | \rp)$. Thus, as anticipated, a second order implicit-explicit time discretization for this kind of problem suffers of the same limitations of standard implicit time discretizations for hyperbolic problems of order higher than one: TVD property is lost. In the sequel, we first study and find a solution to the above problematic
situations
in the simplified setting of linear transport equation and successively we extend the found result to the case of the Euler equations.


\section{Study of the stability on a model problem}\label{model_pb}
We consider the linear advection equation:
\begin{equation}
    \label{eq:model_problem}
    \pt w + c_e \px w + \frac {c_i} {\sqrt{\varepsilon}} \px w = 0,
\end{equation}
with $c_e>0$ and $c_i>0$ fixed real numbers. Note that the dependency in $\sqrt{\varepsilon}$ of the fast velocity is similar to that of the velocity of the pressure waves in the Euler system. The formal limit $\varepsilon\rightarrow 0$ of the above equation is a constant and uniform solution.
The first-order AP scheme detailed in the previous Section becomes
\begin{equation}
    \label{eq:first_order_model}
    w_j^{n+1} = w_j^n
              - \frac {\dt} {\dx} c_e \lp w_j^n - w_{j-1}^n \rp
              - \frac {\dt} {\dx} \frac {c_i} {\sqrt{\varepsilon}} \lp w_j^{n+1} - w_{j-1}^{n+1} \rp.
\end{equation}
The following results hold
\begin{lemma}\label{Prop_first_order_scheme}
For periodic boundary conditions $w^n_0=w^n_L$ and $w^n_{L+1}=w^n_1$ for all $n\geq 0$ and if the following uniform C.F.L. holds true
\begin{equation}
    \label{eq:classical_CFL_condition_model_problem}
    \dt \leq \frac {\dx} {c_e}.
\end{equation}
Then, scheme (\ref{eq:first_order_model}) is Asymptotic Preserving and asymptotically $L^2$- and $L^\infty$-stable, that is
$$    \left \| w^{n+1} \right \|_2 =\lp\sum_{j=1}^L|w_j^{n+1}|^2\rp^{1/2}\leq \left \| w^n \right \|_2
    \text{\qquad and \qquad}
    \left \| w^{n+1} \right \|_\infty =\max_{j=1}^L|w_j^{n+1}|\leq \left \| w^n \right \|_\infty.$$
and TVD, that is
$$    TV(w^{n+1}) \leq TV(w^n)= \sum_{j=1}^L \left| w_{j+1}^n - w_j^n \right|.$$
\end{lemma}
\begin{proof}
The proofs of asymptotically consistency and of the $L^2$- and $L^\infty$-stabilities can be easily established following the results proved in \cite{DimLouVig} and we omit them. It remains to prove the TVD property. Using~(\ref{eq:first_order_model}) and the periodic boundary conditions, we have for all $j=1,\cdots,L$
$$\displaylines{\lp w^{n+1}_{j+1}-w^{n+1}_j\rp\, \lp1+\frac{c_i\, \Delta t}{\sqrt{\eps}\, \Delta x}\rp
-\frac{c_i\, \Delta t}{\sqrt{\eps}\, \Delta x}\, \lp w_j^{n+1}-w^{n+1}_{j-1}\rp   =\lp w^n_{j+1}-w^n_j\rp\,\lp1- \frac {c_e \,\dt} {\dx} \rp + \frac {c_e \,\dt} {\dx}\,\lp w^n_j-w^n_{j-1}\rp.}$$
Taking the absolute value and remarking that for all $a$, $b$ real numbers, $|a|-|b|\leq |a-b|$, summing for all $j=1,\cdots,L$ and
using the periodic boundary conditions and the C.F.L. condition~(\ref{eq:classical_CFL_condition_model_problem}), we obtain
$$\displaylines{
\sum_{j=1}^L\left|w^{n+1}_{j+1}-w^{n+1}_j\right|=\lp1+\frac{c_i\, \Delta t}{\sqrt{\eps}\, \Delta x}\rp\sum_{j=1}^L\left|w^{n+1}_{j+1}-w^{n+1}_j\right|
-\frac{c_i\, \Delta t}{\sqrt{\eps}\, \Delta x}\, \sum_{j=1}^L\left| w_j^{n+1}-w^{n+1}_{j-1}\right|\hfill\cr
\leq \sum_{j=1}^L\left|\lp w^{n+1}_{j+1}-w^{n+1}_j\rp\, \lp1+\frac{c_i\, \Delta t}{\sqrt{\eps}\, \Delta x}\rp
-\frac{c_i\, \Delta t}{\sqrt{\eps}\, \Delta x}\, \lp w_j^{n+1}-w^{n+1}_{j-1}\rp\right|\cr
\hfill   \leq\lp1- \frac {c_e \,\dt} {\dx} \rp\sum_{j=1}^L\left|w^n_{j+1}-w^n_j\right| + \frac {c_e \,\dt} {\dx}\sum_{j=1}^L\left| w^n_j-w^n_{j-1}\right|=\sum_{j=1}^L\left|w^n_{j+1}-w^n_j\right|.
}$$ \end{proof}
We turn now our attention to the two-step ARS(2,2,2) second-order time discretization, we refer to it to as second order AP-scheme. This reads
\begin{subequations}
    \label{eq:ARS222_model}
    \begin{align}
        \label{eq:first_step_ARS222_model}
        &w_j^\star = w_j^n
               - \beta c_e \frac \dt \dx \lp w_j^n - w_{j-1}^n \rp
               - \beta \frac {c_i} {\sqrt{\varepsilon}} \frac \dt \dx \lp w_j^\star - w_{j-1}^\star \rp, \\
        \label{eq:second_step_ARS222_model}
        &\begin{aligned}
        w_j^{n+1} = w_j^n
                 &- ( \beta - 1 ) c_e \frac \dt \dx \lp w_j^n - w_{j-1}^n \rp
                  - ( 1 - \beta ) \frac {c_i} {\sqrt{\varepsilon}} \frac \dt \dx \lp w_j^\star - w_{j-1}^\star \rp \\
                 &- ( 2 - \beta ) c_e \frac \dt \dx \lp w_j^\star - w_{j-1}^\star \rp
                  - \beta         \frac {c_i} {\sqrt{\varepsilon}} \frac \dt \dx \lp w_j^{n+1} - w_{j-1}^{n+1} \rp.
        \end{aligned}
    \end{align}
\end{subequations}
\begin{proposition}
    \label{thm:ARS_is_AP}
    For periodic boundary conditions, the two-step scheme (\ref{eq:ARS222_model}) is asymptotically consistent:
    if $w_j^0=\overline w$ for all $j=1,\cdots,L$, then for all $n\geq 0$, $w_j^n=\overline w$,  for all $j=1,\cdots,L$.
\end{proposition}
\begin{proof}
Multiplying, both equations of~(\ref{eq:ARS222_model}), an passing to the limit $\varepsilon\rightarrow 0$ yields
$w^{\star}_j=w^{\star}_1$ and $w^{n+1}_j=w^{n+1}_1$ for for all $j=1,\cdots,L$.
Now summing~(\ref{eq:second_step_ARS222_model}) for $j=1,\cdots,L$, we obtain by induction
$w^{n+1}_j=w^{n+1}_1=w^{n}_1=\cdots=\overline{w}$.
\end{proof}
Concerning the $L^2$ stability of the scheme (\ref{eq:ARS222_model}), we can prove the following result
\begin{proposition}
    \label{thm:L2_stability}
    For periodic boundary conditions, the scheme (\ref{eq:ARS222_model}) is $L^2$-stable under the C.F.L. condition (\ref{eq:classical_CFL_condition_model_problem}).
\end{proposition}
\begin{proof}
Using Fourier analysis and setting
 $w^{n+1,\star,n}_j=\sum_k\hat w^{n+1,\star,n}_k\, e^{i\, k\, j\,\Delta x}$, we obtain that $w^\star_k=f_k\, w^n_k$
 and $w^{n+1}_k=g_k\, w^n_k$, where $f_k=(1-\beta\, \sigma_e\,(1-c+i\, s))/(1+\beta\,\sigma_i^\eps\, (1-c+i\, s))$ with $\sigma_e=\frac{c_e\, \Delta t}{\Delta x}$,
 $\sigma_i^\eps=\frac{c_i\, \Delta t}{\eps\, \Delta x}$, $c=\cos(k\,\Delta x)$, $s=\sin(k\,\Delta x)$,  and where
$$g_k=\frac{1-(\beta-1)\,\sigma_e\, (1-c+i\,s)}{1+\beta\,\sigma_i^\eps\, (1-c+i\, s)}-
 \frac{\Bigl((1-\beta)\,\sigma_i^\eps+(2-\beta)\sigma_e\Bigl)\, (1-c+i\, s)\,\,(1-\beta\,\sigma_e\, (1-c+i\, s))}{(1+\beta\,\sigma_i^\eps\, (1-c+i\, s))^2}.$$
Remarking that $x\,(1-x)\in[0,1/4]$ for all $x\in[0,1]$, we easily obtain that
under the condition $\beta\, \sigma_e=\beta\, \frac{c_e\, \Delta t}{\Delta x}\leq 1$, we have $ |f_k|^2\leq 1-2\, (\beta\sigma_e)\, (1-c)\, (1-2\,\beta\,\sigma_e)\leq 1$,
 for all $\sigma_i^\eps\geq 0$ and $c\in[-1,1]$.
 Furthermore, an easy calculation shows that $|g_k|^2$ depends only on $s^2$ and has a finite limit when $\sigma_i^\eps\rightarrow +\infty$.
 Then, setting $\sigma_e=1$ and plotting the function $(c,\sigma_i^\eps)\mapsto |g_k|^2$ for $c\in[-1,1]$ and $\sigma_i^{\eps}\in [0,1]$ and
setting $\sigma_e=1$, $\mu_i^\eps=1/\sigma_i^\eps$ and plotting the function $(c,\mu_i^\eps)\mapsto |g_k|^2$ for $c\in[-1,1]$ and $\mu_i^{\eps}\in [0,1]$, we prove that,
for all $\sigma_i^\eps\geq 0$ and $c\in[-1,1]$
 $$\sigma_e=\frac{c_e\, \Delta t}{\Delta x}\leq 1\quad\Rightarrow \quad |g_k|^2\leq 1.$$
\end{proof}
On the other hand, the above scheme is not uniformly $L^\infty$-stable and nor uniformly TVD. Let us see it with a counterexample.
We consider the following initial data on the space domain $[0,1]$
\begin{equation}
    \label{eq:initial_data_advection_step}
    w(0,x) =
    \left\{\begin{array}{ll}
          \varepsilon & \text{if } 0.25 < x \leq 0.75, \\
        - \varepsilon & \text{otherwise.}
    \end{array}\right.
\end{equation}
and on Figure~\ref{fig:advection_step_function}, we display the results of the first-order AP scheme~(\ref{eq:first_order_model}) and the second-order AP scheme~(\ref{eq:ARS222_model}) for two different values of the Mach number using the non-restrictive C.F.L. condition (\ref{eq:classical_CFL_condition_model_problem}) and periodic boundary conditions.
The results are the following: the first-order AP scheme is in-bounds but diffusive, the second-order AP scheme produces bounded spurious oscillations when the time step violates the explicit C.F.L. condition $\Delta x/(c_e+c_i/\sqrt{\varepsilon})$, thus preventing it from being $L^\infty$-stable or TVD independently on $\varepsilon$.
\begin{figure}[!ht]
    \centering
    \includegraphics[height=0.3\textwidth]{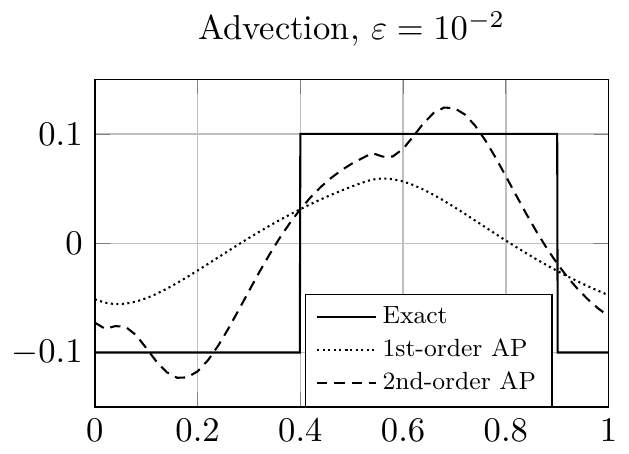} \qquad%
    \includegraphics[height=0.3\textwidth]{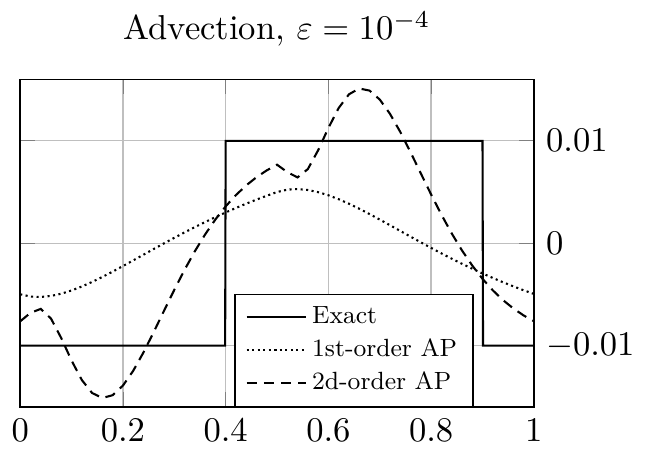}

    \caption{Advection with equation (\ref{eq:model_problem}) of the rectangular pulse~(\ref{eq:initial_data_advection_step}).
             (left panel: $\varepsilon = 10^{-2}$; right panel:~$\varepsilon = 10^{-4}$).
             Comparison of first-order AP scheme (\ref{eq:first_order_model}) (dotted line)
             and the second-order AP scheme (\ref{eq:ARS222_model}) (dashed line)
             against the exact solution (solid line).
             }
    \label{fig:advection_step_function}
\end{figure}
This loss of stability for the case of an implicit-explicit second order scheme shares many similarities with a negative result \cite{GotShuTad} proved in the case of sole implicit high order Runge-Kutta time discretizations that we recall here
\begin{theorem}(\cite{GotShuTad})
    There does not exist TVD implicit Runge-Kutta schemes with unconstrained time steps of order higher than one.
\end{theorem}
To tackle this problem and obtain a TVD numerical scheme that is more accurate than a first-order discretization,
we propose to introduce a convex combination between a first-order implicit-explicit scheme and the IMEX ARS discretization, as follows:
$$w^{n+1}_j=\theta\, w^{n+1,O1}_j+(1-\theta)\, w^{n+1,O2}_j,$$
where $w^{n+1,O1}_j$ is given by the first-order AP scheme~(\ref{eq:first_order_model}), $w^{n+1,O2}_j$ by the second-order AP one~(\ref{eq:ARS222_model}) and $\theta\in[0,1]$.
The spirit is the same of high resolution methods which employ the so-called flux limiter approach \cite{LeV2} for constructing high order TVD schemes. Since, as in the case of high order space discretizations, it is not possible to avoid spurious oscillations, we couple high order discretizations with first order ones and in cases in which the TVD property is violated we come back to the first order discretization which assures monotonicity.
This approach gives the following limited scheme
\begin{subequations}
    \label{eq:theta_model}
    \begin{align}
        \label{eq:theta_model_first_step}
        &w_j^\star = w_j^n
               - \beta c_e \frac \dt \dx \lp w_j^n - w_{j-1}^n \rp
               - \beta \frac {c_i} {\sqrt{\varepsilon}} \frac \dt \dx \lp w_j^\star - w_{j-1}^\star \rp, \\
        \label{eq:theta_model_second_step}
        &\begin{aligned}
            w_j^{n+1} = w_j^n
                     &- \theta ( \beta - 1 )                              c_e \frac \dt \dx \lp w_j^n     - w_{j-1}^n     \rp
                      - \theta ( 1 - \beta ) \frac {c_i} {\sqrt{\varepsilon}} \frac \dt \dx \lp w_j^\star     - w_{j-1}^\star     \rp \\
                     &- \theta ( 2 - \beta )                              c_e \frac \dt \dx \lp w_j^\star     - w_{j-1}^\star     \rp
                      - \theta \beta         \frac {c_i} {\sqrt{\varepsilon}} \frac \dt \dx \lp w_j^{n+1} - w_{j-1}^{n+1} \rp \\
                     &- ( 1 - \theta )                                    c_e \frac \dt \dx \lp w_j^n     - w_{j-1}^n     \rp
                      - ( 1 - \theta )       \frac {c_i} {\sqrt{\varepsilon}} \frac \dt \dx \lp w_j^{n+1} - w_{j-1}^{n+1} \rp.
        \end{aligned}
    \end{align}
\end{subequations}
For which the following results hold true:
\begin{theorem}
    \label{thm:AP_theta_scheme}
    With periodic boundary conditions, scheme (\ref{eq:theta_model}) is asymptotically consistent.
\end{theorem}
\begin{proof}
    The scheme (\ref{eq:theta_model}) is a result of a convex combination of two asymptotically consistent schemes (\ref{eq:first_order_model}) and (\ref{eq:ARS222_model}).
    Therefore, it is also asymptotically consistent.
\end{proof}
\begin{theorem}
    \label{thm:TVD_scheme}
    The scheme (\ref{eq:theta_model}) is asymptotically stable, i.e. uniformly TVD and $L^\infty$-stable:
$$ \left \| w^{\star} \right \|_\infty \leq \left \| w^n \right \|_\infty,\quad\quad\left \| w^{n+1} \right \|_\infty \leq \left \| w^n \right \|_\infty,\quad\quad TV(w^{\star}) \leq TV(w^n),\quad\quad TV(w^{n+1}) \leq TV(w^n),$$
if the following uniform C.F.L. conditions
    \begin{equation}
        \label{eq:bound_on_theta_and_CFL}
        (1-\alpha)\,\frac{ c_e \, \dt}{\dx}\leq 1 - \alpha,\quad\quad
 \alpha\, \frac{ c_e \, \dt}{\dx}\leq \alpha\,\frac{1}{ \frac{\beta}{ 1 - \beta } ( 2 - \beta )}=\alpha\,\sqrt{2}.
    \end{equation}
are verified for $\alpha\in[0,1]$ and $\theta =\alpha\, \frac \beta { 1 - \beta } = \alpha\, (\sqrt{2} - 1)\in]0,1[,$
\end{theorem}
\begin{proof}
    Let us first note that, like for the first-order AP scheme, the first step (\ref{eq:theta_model_first_step}) of the scheme is TVD and $L^\infty$-stable
    under the non-restrictive explicit CFL condition $c_e\,\dt/\dx \leq 1/  \beta$ which is true since $1\leq\sqrt{2} \leq 1/  \beta $.
    Therefore, $\left \| w^\star \right \|_\infty \leq \left \| w^n \right \|_\infty$ and $TV \lp w^\star \rp \leq TV \lp w^n \rp$.

Let us prove  the $L^\infty$ stability of the second step. We denote by $j_0$ the index in $\{1,\cdots,L\}$ such that $w_{j_0}^{n+1}=\max_{j=1}^Lw_j^{n+1}$.
Thanks to the periodic boundary conditions we have $w_{j_0}^{n+1}=\max_{j=0}^{L+1}w_j^{n+1}$. Then,
we rewrite the second step (\ref{eq:theta_model_second_step}) as follows:
$$\displaylines{w_{j_0}^{n+1}\leq w_{j_0}^{n+1}+(1-\theta+\theta\, \beta)\, \frac{c_i\, \Delta t}{\sqrt{\eps}\, \Delta x} (w_{j_0}^{n+1}-w_{j_0-1}^{n+1}) =
                     \ w_{j_0}^n \lp 1 - \theta ( \beta - 1 ) \frac{ c_e \, \dt}{\dx} - ( 1 - \theta )\frac{ c_e \, \dt}{\dx} \rp\hfill\cr
                  +  \ w_{j_0-1}^n \lp \theta ( \beta - 1 ) \frac{ c_e \, \dt}{\dx} + ( 1 - \theta ) \frac{ c_e \, \dt}{\dx} \rp \cr
                  \hfill-  \ \theta ( 1 - \beta )  \frac {c_i\,  \dt } {\sqrt{\varepsilon}\, \dx} ( w_{j_0}^\star - w_{j_0-1}^\star ) - \theta ( 2 - \beta ) \frac{ c_e \, \dt}{\dx} ( w_{j_0}^\star - w_{j_0-1}^\star ).}$$
    From the first step (\ref{eq:theta_model_first_step}), we deduce
    $- \frac {c_i\,  \dt } {\sqrt{\varepsilon}} ( w_{j_0}^\star - w_{j_0-1}^\star ) =\frac{1}{ \beta} ( w_{j_0}^\star - w_{j_0}^n ) + \frac{ c_e \, \dt}{\dx} ( w_{j_0}^n - w_{j_0-1}^n )$.
    Plugging this expression into the previous equality leads to:
  $$\displaylines{ \max_{j=0}^{L+1}w_j^{n+1}=w_{j_0}^{n+1}\leq
                       \ w_{j_0}^n \Biggl( 1 - \theta\  \frac { 1 - \beta }{\beta}- \frac{ c_e \, \dt}{\dx}
                      \Bigl(1-\theta\,  \overbrace{(3-2\,\beta)}^{=(1-\beta)/\beta}\Bigl)\Biggl) +  \ w_{j_0-1}^n\ \frac{ c_e \, \dt}{\dx}\lp1-\theta\, \frac { 1 - \beta }{\beta}\rp\cr
          +  \ w_{j_0}^\star \lp \theta \frac { 1 - \beta } \beta - \theta ( 2 - \beta )  \frac{ c_e \, \dt}{\dx} \rp + w_{j_0-1}^\star \theta ( 2 - \beta ) \frac{ c_e \, \dt}{\dx}.}$$
 Note that if all coefficients are positive, we will have two convex combinations, one at time index $n$ and one at time index $\star$, and
 $$ \max_{j=0}^{L+1}w_j^{n+1}=w_{j_0}^{n+1}\leq  \max_{j=0}^{L+1}w_j^{n}\,\lp1 - \theta\  \frac { 1 - \beta } \beta\rp+\theta\  \frac { 1 - \beta } \beta\, \max_{j=0}^{L+1}w_j^{\star}
 \leq \max_{j=0}^{L+1}w_j^{n}.$$
From the above expression we deduce that necessary conditions for having positive coefficients are
 $1 - \theta\  \frac { 1 - \beta } \beta\geq0$ and $\theta \frac { 1 - \beta } \beta\geq 0$.
This
gives the condition
 $\theta\in[0,\beta/(1-\beta)]$. By setting
 $\theta=\alpha\, \frac{\beta}{ 1 - \beta },$
 with $\alpha\in[0,1]$,
then
all coefficients are positive under the C.F.L. conditions (\ref{eq:bound_on_theta_and_CFL}) and consequently the $L^\infty$ property is verified.

We can now prove the TVD property. Using~(\ref{eq:theta_model_second_step}) and the periodic boundary conditions
and using the first step, we have for all $j=1,\cdots,L$
$$\displaylines{\lp w^{n+1}_{j+1}-w^{n+1}_j\rp\, \lp1+(1-\theta+\theta\, \beta)\,\frac{c_i\, \Delta t}{\sqrt{\eps}\, \Delta x}\rp
-(1-\theta+\theta\, \beta)\,\frac{c_i\, \Delta t}{\sqrt{\eps}\, \Delta x}\, \lp w_j^n-w^n_{j-1}\rp\hfill\cr
= \lp w_{j+1}^n-w_{j}^n\rp \Biggl( 1 - \theta\  \frac { 1 - \beta } \beta- \frac{ c_e \, \dt}{\dx}
                      \Bigl(1+\theta\,  (1-2\,\beta)\Bigl)\Biggl) +  \ \lp w_{j}^n-w_{j-1}^n\rp\ \frac{ c_e \, \dt}{\dx}\Bigl(1+\theta\, (1-2\,\beta)\Bigl)\cr
          +  \ \lp w_{j+1}^\star-w_{j}^\star\rp \lp \theta \frac { 1 - \beta } \beta - \theta ( 2 - \beta )  \frac{ c_e \, \dt}{\dx} \rp + \lp w_{j}^\star-w_{j-1}^\star\rp \theta ( 2 - \beta ) \frac{ c_e \, \dt}{\dx}.}$$
Taking the absolute value, remarking that for all $a$, $b$ real numbers, $|a|-|b|\leq |a-b|$, summing for all $j=1,\cdots,L$ and
using the C.F.L. condition~(\ref{eq:bound_on_theta_and_CFL}) and the periodic boundary conditions, we conclude the proof:
$$\displaylines{\sum_{j=1}^L|w^{n+1}_{j+1}-w^{n+1}_j|
\leq  \sum_{j=1}^L|w_{j+1}^n-w_{j}^n| \Biggl( 1 - \theta\  \frac { 1 - \beta } \beta- \frac{ c_e \, \dt}{\dx}
                      \Bigl(1+\theta\,  (1-2\,\beta)\Bigl)\Biggl) \hfill\cr
                      \hfill+  \sum_{j=1}^L| w_{j}^n-w_{j-1}^n|\ \frac{ c_e \, \dt}{\dx}\Bigl(1+\theta\, (1-2\,\beta)\Bigl)\cr
          +  \sum_{j=1}^L|w_{j+1}^\star-w_{j}^\star| \lp \theta \frac { 1 - \beta } \beta - \theta ( 2 - \beta )  \frac{ c_e \, \dt}{\dx} \rp +\sum_{j=1}^L|w_{j}^\star-w_{j-1}^\star| \theta ( 2 - \beta ) \frac{ c_e \, \dt}{\dx}\leq \sum_{j=1}^L|w_{j+1}^n-w_{j}^n|.
          }$$
\end{proof}

\begin{remark}
\begin{enumerate}
\item Theorem~\ref{thm:TVD_scheme} shows that the largest possible value for $\theta$ is $\theta_M=\beta/(1-\beta)=\sqrt{2}-1\approx 0.4142$ if TVD and $L^\infty$ properties
need
to be assured. We refer to as TVD-AP scheme when discussing of scheme (\ref{eq:theta_model}) with $\theta=\theta_M$. This value depends on the type of second order time discretization chosen. Other time discretizations, may allow larger values which
may
possibly improve the accuracy of the method. This
situation will be discussed in detail in \cite{DLV}.
\item In order to additionally increase the accuracy of the method one
introduces local values for $\theta$ different for each spatial cell in equation~(\ref{eq:theta_model}). However, in this case, the proof of the TVD as well as of the $L^\infty$ stability remain an open problem. In addition, numerical experiments performed suggest that such a local parameter must be chosen related to a stencil of neighbors which is proportional to the velocity $c_i/\sqrt{\eps}$. These aspects will be studied in detail in \cite{DLV}.
 \end{enumerate}
\end{remark}
We discuss now limiters which allow to detect the situations in which the TVD or the $L^\infty$ stability property is violated and consequently to
switch
from the second order accurate scheme to the TVD-AP scheme without
loosing excessive
accuracy by diminishing parameter $\theta$. The proposed approach is based on a detection technique borrowed from the MOOD (Multidimensional Optimal Order Detection) \cite{ClaDioLou,DioClaLou,DioLouCla} framework. The idea behind this specific MOOD approach is to use the second-order oscillatory discretization (\ref{eq:ARS222_model}) whenever possible, i.e. when no oscillations appear. Instead, if at time $n$ the numerical solution presents oscillations, we discard it and we replace it by the limited TVD-AP scheme, i.e. scheme (\ref{eq:theta_model}) with $\theta=\theta_M=\beta/(1-\beta)$ which assures preservation of the demanded properties.
Since, for this specific situation, the $L^\infty$ norm of the solution is preserved in time, spurious
oscillations are checked with respect to the initial condition instead of that of the previous time iteration of the scheme. Indeed, the relevant bounds are those of the initial condition rather than the ones of the diffusive numerical approximation.
The procedure can be summarized by the following algorithm
\begin{algorithm}
    \label{thm:algorithm_MOOD_model_problem}
    \begin{enumerate}
        \item Compute a candidate solution $w^{n+1,O2}$ using the second-order AP scheme~(\ref{eq:ARS222_model}).
        \item \label{item:detection}
              Detect if this candidate solution satisfies the $L^\infty$ stability and the TVD property.
        \item \label{item:scheme_choice}
              If these two criteria are satisfied by $w^{n+1,O2}$, set $w^{n+1} = w^{n+1,O2}$;
              otherwise, compute $w^{n+1}$ using the TVD-AP scheme~(\ref{eq:theta_model}) with $\theta =\theta_M= \beta / ( 1 - \beta )$.
    \end{enumerate}
\end{algorithm}
We refer to as the AP-MOOD scheme when this Algorithm is used.
In Figure~\ref{fig:advection_step_function_MOOD}, we report the results of the advection of the rectangular pulse given by (\ref{eq:initial_data_advection_step}), for different values of the parameter $\varepsilon$. The solution is computed by the first order AP scheme, the second order AP scheme, the TVD-AP scheme and the AP-MOOD scheme. The exact solution is also reported. One can clearly see the differences in terms of accuracy for the different methods proposed and the absence of spurious oscillations for the TVD-AP and the AP-MOOD methods. In the sequel, we will extend this approach to the case of the Euler equations.
\begin{figure}[!ht]
    \centering
    \includegraphics[height=0.3\textwidth]{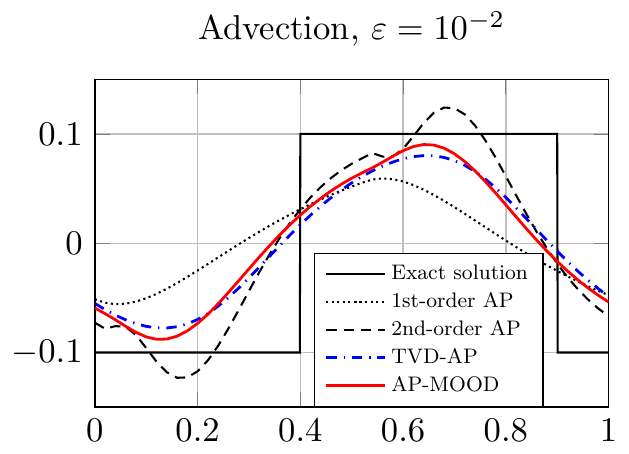} \qquad%
    \includegraphics[height=0.3\textwidth]{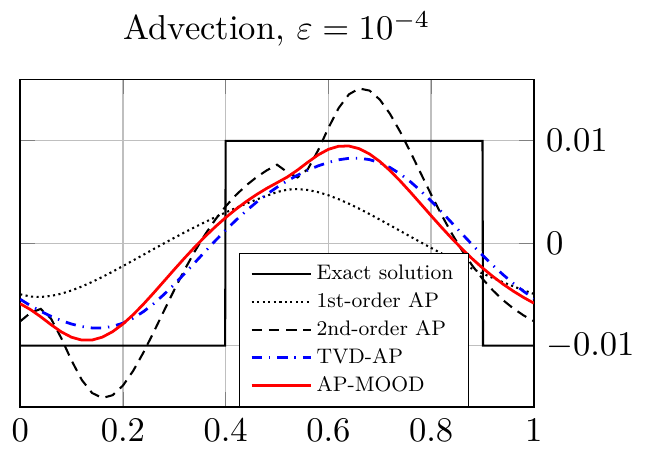}
    \caption{Advection with equation (\ref{eq:model_problem}) of the rectangular pulse~(\ref{eq:initial_data_advection_step}).
             (left panel: $\varepsilon = 10^{-2}$; right panel:~$\varepsilon = 10^{-4}$).
             Comparison of the first-order AP scheme (\ref{eq:first_order_model}) (dotted line),
             the second-order AP scheme (\ref{eq:ARS222_model}) (dashed line), the TVD-AP scheme~(\ref{eq:theta_model}) (red line) and
             of AP-MOOD scheme given by Algorithm~\ref{thm:algorithm_MOOD_model_problem} (blue line)
             against the exact solution (solid line).
    }
    \label{fig:advection_step_function_MOOD}

\end{figure}


\section{Application to the isentropic Euler system}\label{APEuler}
We now extend the idea developed in the previous Section to the isentropic Euler system.
The TVD-AP scheme reads as
\begin{equation}\label{Euler_limited_AP_scheme}
W^{n+1}_j=\theta\, W^{n+1,O1}_j+(1-\theta)\, W^{n+1,O2}_j,
\end{equation}
where $W^{n+1,O1}_j$ is given by the first-order AP scheme~(\ref{eq:first_order_scheme_fully_discrete})
while $W^{n+1,O2}_j$ by the second-order AP scheme~(\ref{semi_ARS_uncoupled}) and $\theta$ is fixed equal to $\theta_M=\beta/(1-\beta)$. This is enough to assure the TVD and the $L^\infty$
stability properties. However, since we observed that in many situations the full second order AP scheme (\ref{semi_ARS_uncoupled}) can be employed
without formation of spurious oscillations, as for the case of the linear advection described before, we aim in constructing a MOOD like technique which
permits to interpolate from the full second order to the TVD-AP scheme if needed, producing
\textit{de facto}
an highly accurate method which is referred to as the AP-MOOD method.
Unfortunately, in this case, one can not directly transpose to the Euler case the MOOD approach seen for the advection equation. In fact, in the Euler system, the variables $\rho$ and $q$ no longer satisfy either the TVD property or the $L^\infty$ bound in the continuous case. As a consequence, we cannot apply the detection criteria seen before on the conservative variables $\rho$ and $q$ to get a non oscillating scheme. It turns
out
that characteristic or Riemann invariants variables constitute a better choice for detecting spurious oscillations since it can be shown that they verify some decoupled non linear advection equations \cite{Toro}. We denote the Riemann invariants by $\phi_+$ and $\phi_-$. For the isentropic Euler system case, they are given by
\begin{equation}
   \label{eq:Riemann_invariants}
   \phi_+(W) := \frac q \rho - h(\rho),
   \text{\qquad and \qquad}
   \phi_-(W) := \frac q \rho +h(\rho),
\end{equation}
where $h(\rho)$ is the enthalpy given by $h(\rho)=\frac{2}{\gamma-1}\,\sqrt{\frac{\gamma\,\rho^{\gamma-1}}{\eps}}$ if $\gamma>1$ and $h(\rho)=\ln(\rho)$ if $\gamma=1$ \cite{Toro}.
Now, since it is known that at continuous level and for a Riemann problem at least one of the two Riemann invariants $\phi_+$ or $\phi_-$ satisfy the maximum principle \cite{Smo}, one can think to introduce a MOOD-like detection criterion which relies on testing whether both Riemann invariants break the maximum principle at the same time.
In practice, the following stability detector is used
\begin{equation*}
    \label{eq:definition_bound_Linf_norm}
    \begin{array}{l}
        M_\pm^0 = \| \Phi_\pm^0 \|_\infty, \\
        M_\pm^n = \max \lp M_\pm^{n-1}, \| \Phi_\pm^n \|_\infty \rp, \text{ for all } n > 0.
    \end{array}
\end{equation*}
Equipped with this detector, the AP-MOOD algorithm for the Euler equations reads as follows
\begin{algorithm}
    \label{thm:algorithm_MOOD_Euler}
    \begin{enumerate}
        \item Compute a candidate solution $W^{n+1,O2}$ using the second order AP scheme~(\ref{semi_ARS_uncoupled}).
        \item Detect
          if this candidate solution satisfies the maximum principle of the Riemann invariants:
        $\| \Phi_-^{n+1,O2} \|_\infty \leq M_-^n$ and $\| \Phi_+^{n+1,02} \|_\infty \leq M_+^n$;
        \item If these two criteria are satisfied, set $W^{n+1} = W^{n+1,O2}$.
              otherwise, compute $W^{n+1}$ using the TVD-AP scheme~(\ref{Euler_limited_AP_scheme}).
    \end{enumerate}
\end{algorithm}
With this approach, at most one extra computation of the semi-implicit scheme is required to ensure the TVD property and the $L^\infty$ stability.

We now turn to a second-order discretization in space. We present it in the case of one space dimension for the sake of simplicity.
To this end, we use classical MUSCL techniques \cite{VLe}, other high order space reconstruction could be employed as well without changing the core of the method.
The enunciated discretization works by introducing a linear reconstruction of the conserved variables~$W_j^n$:
\begin{equation}
    \label{eq:linear_reconstruction}
    \widehat W_j^n(x) = W_j^n + \sigma_j^n ( x - x_j ),
\end{equation}
where $\sigma_j^n$ can be a limited (if TVD property should be assured) or unlimited slope. The case of unlimited slope is used in combination with the second order in time AP scheme (\ref{semi_ARS_uncoupled}) and it is given by
\begin{equation}
    \label{eq:unlimited_slope}
    \sigma_j^n = \frac 1 2 \lp \frac { W_j^n - W_{j-1}^n } \dx + \frac { W_{j+1}^n - W_j^n } \dx \rp.
\end{equation}
This gives rise to a genuine second order in time and space Asymptotic Preserving method which however does not enjoy the TVD and $L^\infty$ property.
On the other hand, the limited slope for which the \minmod limiter is used, is employed together with AP-TVD scheme in time (\ref{Euler_limited_AP_scheme}). In this case we have
\begin{equation}
    \label{eq:minmod_limited_slope}
    \sigma_j^n = \minmod \lp \frac { W_j^n - W_{j-1}^n } \dx, \frac { W_{j+1}^n - W_j^n } \dx \rp,
\end{equation}
where the \minmod function is given by:
$$    \minmod(a,b) =
    \left\{\begin{array}{ll}
        \min(a,b) & \text{if } a > 0 \text{ and } b > 0, \\
        \max(a,b) & \text{if } a < 0 \text{ and } b < 0, \\
                0 & \text{otherwise.}
    \end{array}\right.
$$
The above combination of space and time discretization gives rise to a TVD-AP highly accurate space and time discretization of the Euler equation.
The reconstruction of variables $W$ (\ref{eq:linear_reconstruction}) is used for defining the numerical flux functions at the interfaces
\begin{equation}
    \label{eq:notations_inner_interfaces}
    W_{j,\pm}^n := \widehat W_j^n \lp x_j \pm \frac \dx 2 \rp = W_j^n \pm \frac \dx 2 \sigma_j^n,
\end{equation}
and thus, the explicit numerical flux function $\F_e$ becomes
\begin{equation}
    \label{eq:explicit_flux_MUSCL}
      (\F_e)_{\jph}^n := \frac {F_e(W_{j,+}^n) + F_e(W_{j+1,-}^n)} 2 + (\D_e)_{\jph}^n ( W_{j+1,-}^n - W_{j,+}^n ),
\end{equation}
with $(\D_e)_{\jph}^n:=  \max \lp |u_{j,+}^n|,|u_{j+1,-}^n| \rp$, while the implicit numerical flux function $\F_i$ becomes
\begin{equation}
    \label{eq:implicit_flux_MUSCL}
   (\F_i)_{\jph}^{n,n+1} :=  \frac {F_i(W_{j,+}^{n,n+1}) + F_i(W_{j+1,-}^{n,n+1})} 2 + (\D_i)_{\jph}^n ( \tilde W_{j+1,-}^{n+1} - \tilde W_{j,+}^{n+1} ),
\end{equation}
where $W_{j,+}^{n,n+1}=(\rho^{n+1}_{j,+}, q_{j,+}^n)$ and where $(\D_i)_{\jph}^n$ is the implicit viscosity coefficient, taken as half of the maximum implicit eigenvalue and given by
\begin{equation}
(\D_i)_{\jph}^n := \frac 1 2 \max \lp \sqrt{ \dfrac {p'(\rho_{j,+}^n)} \varepsilon }, \sqrt{ \dfrac {p'(\rho_{j+1,-}^n)} \varepsilon } \rp,
\end{equation}
and
\begin{equation}
    \label{eq:linear_reconstruction_np1}
    \tilde W_{j,\pm}^{n+1} = W_j^{n+1} \pm \frac \dx 2 \sigma_j^n.
\end{equation}
The second-order MUSCL extension in space is thus complete.


\section{Numerical experiments}\label{numeric}
The schemes described in the previous parts are resumed and labeled below.
\begin{itemize}
    \item The \emph{first-order AP scheme} is given by (\ref{eq:first_order_scheme_fully_discrete}),~(\ref{eq:explicit_flux}),~(\ref{eq:implicit_flux}).
    \item The \emph{second-order AP scheme} is given by~(\ref{semi_ARS_uncoupled}),~(\ref{eq:explicit_flux_MUSCL}),~(\ref{eq:implicit_flux_MUSCL})
          corresponding to the ARS(2,2,2) time discretization, supplemented with
          the unlimited linear MUSCL
          space
          reconstruction.
    \item The \emph{\tap scheme} is given by~(\ref{Euler_limited_AP_scheme}) with $\theta = \beta / ( 1 - \beta )$
        corresponding to convex combination of the first-order AP scheme and of the second-order AP scheme with (\ref{eq:explicit_flux_MUSCL}),~(\ref{eq:implicit_flux_MUSCL})
        and limiter (\ref{eq:minmod_limited_slope}) for the space discretization.
    \item The \emph{\tapm scheme} corresponds to the procedure detailed in Algorithm~\ref{thm:algorithm_MOOD_Euler} with (\ref{eq:explicit_flux_MUSCL}),~(\ref{eq:implicit_flux_MUSCL})
        and limiter (\ref{eq:minmod_limited_slope}) for the space discretization.
\end{itemize}
In addition, for all the schemes the time step is constrained by the uniform C.F.L. condition
\begin{equation*}
    \label{eq:CFL_numerics}
    \dt \leq C \frac \dx \Lambda, \text{\qquad where } \Lambda = \max_j \lp 2 \,|u_j^n| \rp,
\end{equation*}
with $C = 0.9$ for the first-order scheme and $C = 0.45$ for the other three schemes. Note that this restrictive C.F.L. (for the second-order AP, TVD-AP and AP-MOOD schemes) is uniform and is only due to the second order discretization in space. In the following, we first consider a Riemann problem and successively perform an assessment of the order of accuracy of the scheme using a smooth solution in one space dimension. Afterwards, we validate the scheme on a more complex test case and we
verify
its asymptotic stability again in one space dimension.
Finally, we propose two two-dimensional numerical experiments in which we compare our scheme with a reference solution.


\subsection{One dimensional shock tube}\label{Sod}

On the space domain $[0,1]$, we consider $\gamma=1.4$ and a Riemann problem with the following initial data:
$$    \label{eq:initial_condition_rarefaction_shock}
    \rho(0,x) =\left\{
    \begin{array}{ll}
        1 + \varepsilon &\text{if } x < 0.5, \\
                      1 &\text{otherwise;}
    \end{array}\right.
    \qquad \quad
    q(0,x) = 1.
$$
Homogeneous Neumann boundary conditions are prescribed on each boundary.
We compare the results from the three schemes in several regimes corresponding to different values of the Mach number: $\varepsilon = 1$, $\varepsilon = 10^{-2}$ and $\varepsilon = 10^{-4}$.
The results are displayed in Figure~\ref{fig:rarefaction_shock_MOOD_eps_1} for $\varepsilon = 1$ and $N=50$, in Figure~\ref{fig:rarefaction_shock_MOOD_eps_1em2} for $\varepsilon = 10^{-2}$ and $N=125$ and in Figure~\ref{fig:rarefaction_shock_MOOD_eps_1em4} for $\varepsilon = 10^{-4}$ and $N=500$ on the left for the density and on the right for the momentum. As expected, the first-order AP scheme (dotted line) is very diffusive. On the other hand, the second-order scheme (dashed line) yields a better approximation of the intermediate states. However, overshoots and undershoots appear at the heads and tails of the rarefaction wave, and near the shock wave. The \tap scheme (blue line) corrects both of these shortcomings. Finally, thanks to the MOOD procedure, the \tapm scheme (blue dashed line) yields better results than the \tap scheme.
\begin{figure}[!ht]
    \centering
    \includegraphics[height=0.3\textwidth]{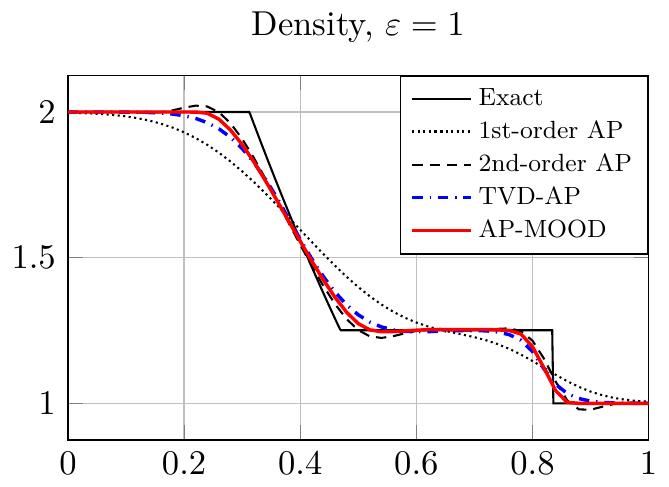} \qquad%
    \includegraphics[height=0.3\textwidth]{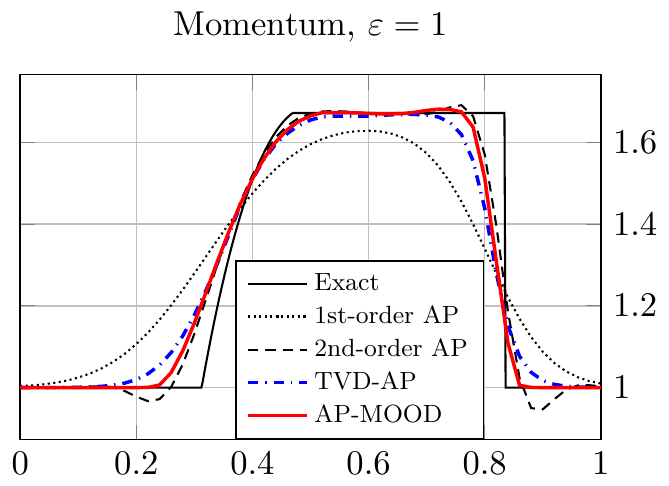}
    \caption{Shock tube with $\varepsilon = 1$ and $50$ discretization cells; results displayed at time $t_{end} = 0.125$.}
    \label{fig:rarefaction_shock_MOOD_eps_1}
\end{figure}
\begin{figure}[!ht]
    \centering
    \includegraphics[height=0.3\textwidth]{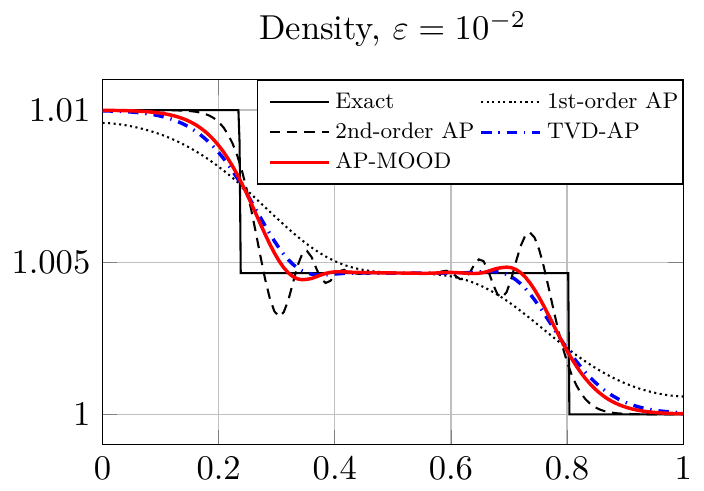} \qquad%
    \includegraphics[height=0.3\textwidth]{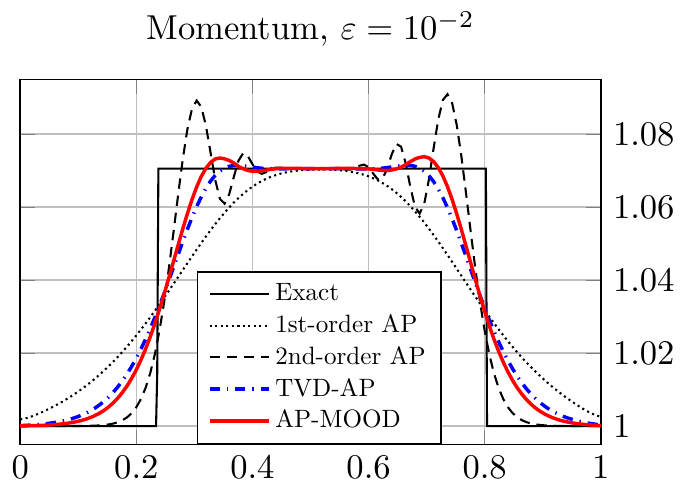}
    \caption{Shock tube with $\varepsilon = 10^{-2}$ and $125$ discretization cells, displayed at time $t_{end} = 0.02$. }
    \label{fig:rarefaction_shock_MOOD_eps_1em2}
\end{figure}
\begin{figure}[!ht]
    \centering
    \includegraphics[width=0.475\textwidth]{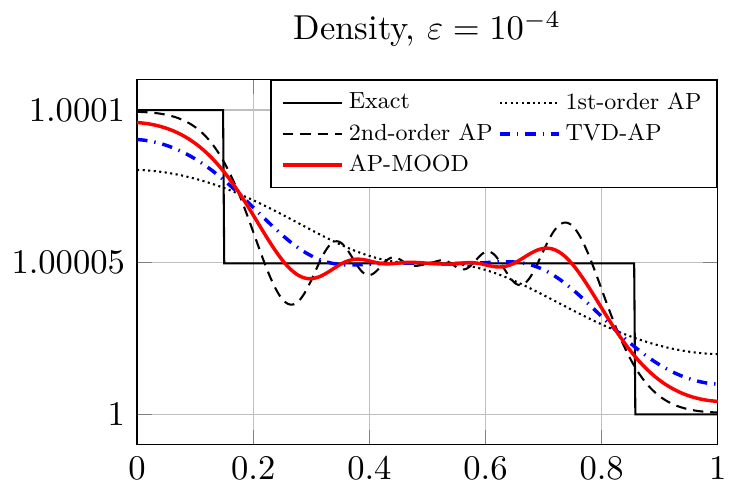} \qquad%
    \includegraphics[width=0.45\textwidth]{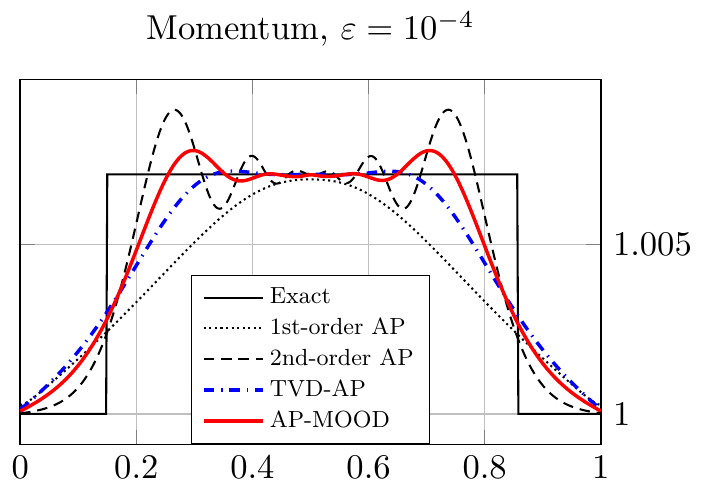}
    \caption{Shock tube with $\varepsilon = 10^{-4}$ and $500$ discretization cells, displayed at time $t_{end} = 0.0025$. }
    \label{fig:rarefaction_shock_MOOD_eps_1em4}

\end{figure}
In conclusion, the oscillatory nature of the second-order AP scheme is removed by the \tap and \tapm schemes at the cost of an expected slightly increased diffusion.

\subsection{Order of accuracy assessment in one dimension}
We consider a smooth solution from \cite{VilMaiAbg} with the following initial data
   $$ \rho(0,x) = 1 - \dfrac \varepsilon 2 \, \omega \lp \frac 2 {0.25} \lp x - \frac 1 2 \rp \rp
    \text{\qquad and \qquad}
    u(0,x) = 1 + \dfrac \varepsilon 2 \, \omega \lp \frac 2 {0.25} \lp x - \frac 1 2 \rp \rp,$$
on the space domain $[0,1]$ and where the function $\omega$ is given by $\omega(z) =\lp\frac { 2 - |z| } 2 \rp^4 ( 1 + 2 |z| )$ if  $|z| \leq 2$, and  $0$ otherwise.
If $\gamma=3$, see~\cite{VilMaiAbg}, both Riemann invariants are solution to the following Burgers equations:
$$    \begin{cases}
        \pt \phi_+ + \phi_+ \px \phi_+ = 0, \\
        \pt \phi_- + \phi_- \px \phi_- = 0.
    \end{cases}
$$
Solving the above system requires a nonlinear equation solver, such as Newton's method.

For small enough time, the exact solution $(\rho, q)$ is as smooth as the initial data. We use it
to determine the Dirichlet boundary conditions for the four schemes and to compute
 the errors between the approximate solutions and the exact solution.
We measure the $L^\infty$ errors for the density and the momentum
$$ e_\infty^n(\rho) = \max_j \left| \rho_j^n - (\rho_{ex})_j^n \right|, \ e_\infty^n(q) = \max_j \left| q_j^n - (q_{ex})_j^n \right|$$
where $^t( (\rho_{ex})_j^n, (q_{ex})_j^n )$ is the exact solution at time $t^n$ in the cell of center $x_j$.
The time at which the errors are computed are $t_{end} = 0.007$ for $\varepsilon = 1$, $t_{end} = 0.005$ for $\varepsilon = 10^{-2}$ and $t_{end} = 0.0005$ for $\varepsilon = 10^{-4}$.
For the four schemes, the density and momentum $L^\infty$  errors are displayed in Figure~\ref{fig:error_lines} in logarithmic scale with respect to the number of discretization cells.
\begin{figure}[!ht]
    \centering
    \includegraphics[width=0.45\textwidth]{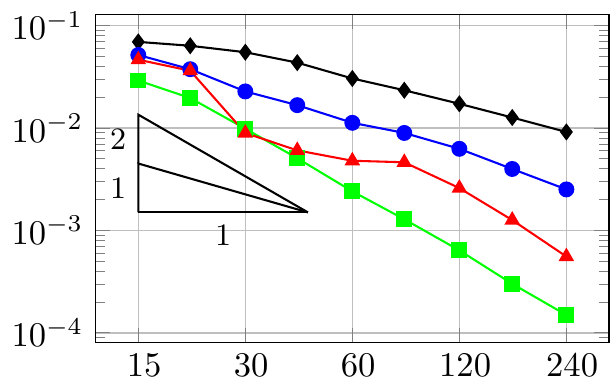} \qquad%
    \includegraphics[width=0.45\textwidth]{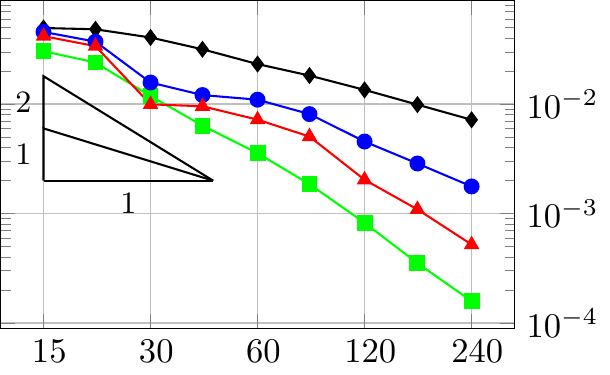}
    \medskip
    \includegraphics[width=0.45\textwidth]{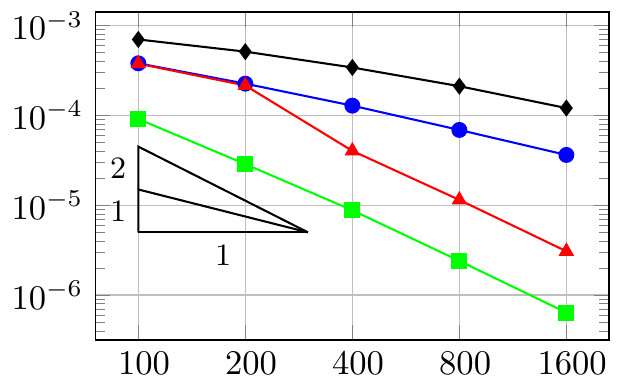} \qquad%
    \includegraphics[width=0.45\textwidth]{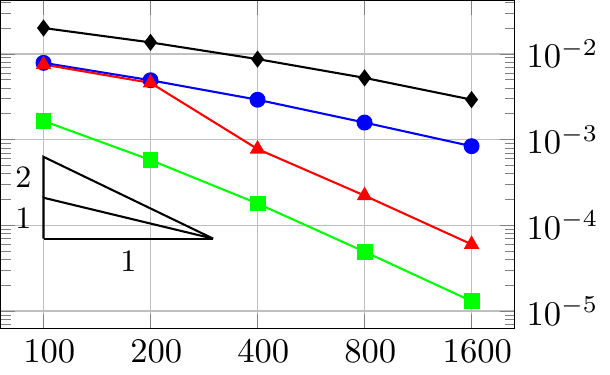}
    \medskip
    \includegraphics[width=0.45\textwidth]{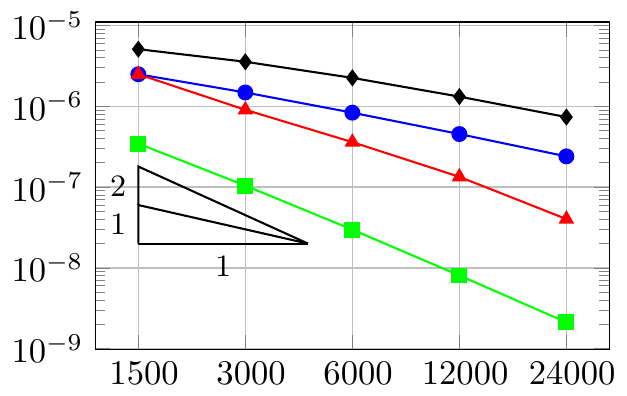} \qquad%
    \includegraphics[width=0.45\textwidth]{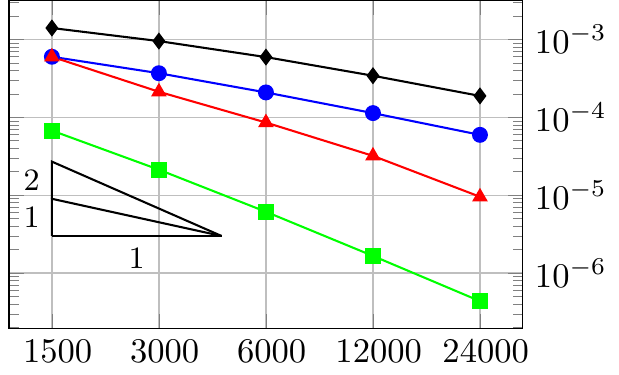}
    \caption{Error lines. From top to bottom: $\varepsilon = 1$, $\varepsilon = 10^{-2}$ and $\varepsilon = 10^{-4}$.
             Left panels: density errors; right panels: momentum errors. Black line corresponds to the first order AP scheme, blue line to the \tap scheme,
             red line to the \tapm scheme, yellow line to the second order AP unlimited.}
    \label{fig:error_lines}

\end{figure}
For all values of $\varepsilon$, the first-order AP scheme and the second-order AP scheme are respectively of order $1$ and $2$, as expected.
In addition, we note that the \tap scheme is also numerically first-order accurate or barely larger but with an $L^\infty$  error which is always lower than the one of the first order method.
The \tapm for $\varepsilon \in \{ 10^{-2}, 10^{-4} \}$
is numerically of order two in spite of the slope limiter.
For $\varepsilon = 1$, the \tapm scheme is of order more than one but less than two with an $L^\infty$  error always smaller than the one of the first order method.

\subsection{Validation and asymptotic stability in one dimension}
We now consider the problem introduced in
Degond and Tang
\cite{DegTan}. It consists in several interacting Riemann problems.
The initial data are given on the space domain~$[0,1]$ by
\begin{equation*}
    \label{eq:initial_condition_Degond_Tang}
    \rho(0,x) =
    \left\{ \begin{array}{ll}
        2                & \text{ if }  x\in[0,0.2], \\
        2+\varepsilon    & \text{ if }  x\in(0.2,0.3], \\
        2                & \text{ if }  x\in(0.3,0.7] , \\
        2-\varepsilon    & \text{ if }  x\in(0.7,0.8), \\
        2                & \text{ if }  x\in[0.8, 1],
    \end{array} \right.
    \text{\quad and \quad}
    q(0,x) =
    \left\{ \begin{array}{ll}
        1-\varepsilon/2  & \text{ if }  x\in[0,0.2], \\
        1                & \text{ if }  x\in(0.2,0.3], \\
        1 +\varepsilon/2 & \text{ if }  x\in(0.3,0.7] , \\
        1                & \text{ if }  x\in(0.7,0.8), \\
        1-\varepsilon/2  & \text{ if }  x\in[0.8, 1].
    \end{array} \right.
\end{equation*}
supplemented by periodic boundary conditions. We choose $\gamma=1.4$. Here, the goal is to validate the proposed schemes in both the compressible and the incompressible regimes.
The reference solution is computed with the first order-AP scheme
on a refined mesh in space and time.
Figure~\ref{fig:Degond_Tang_eps_1} reports the results for the density on the left and the momentum on the right for $\varepsilon = 1$ and $N=100$ discretization cells with final time $t_{end} = 0.075$.
In the top panels of Figure~\ref{fig:Degond_Tang_eps_1em4}, we display the solution for $\varepsilon = 10^{-4}$ and $t_{end} = 0.0015$ obtained with $1500$
cells.
In the bottom panels, we have refined the space-time mesh to study the convergence of the numerical approximations.
\begin{figure}[!ht]
    \centering
    \includegraphics[height=0.33\textwidth]{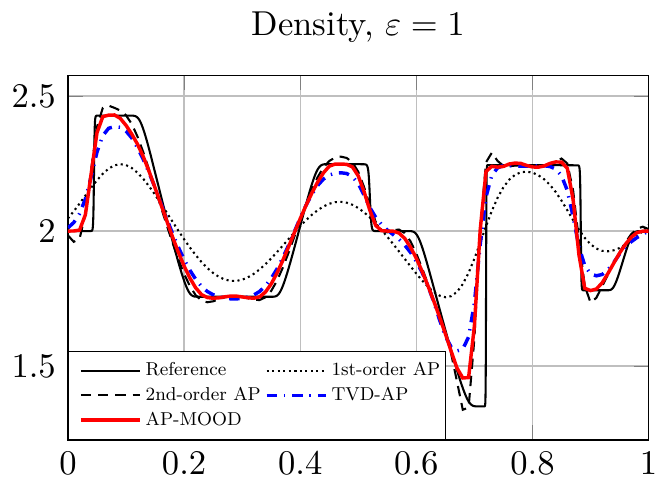} \qquad%
    \includegraphics[height=0.33\textwidth]{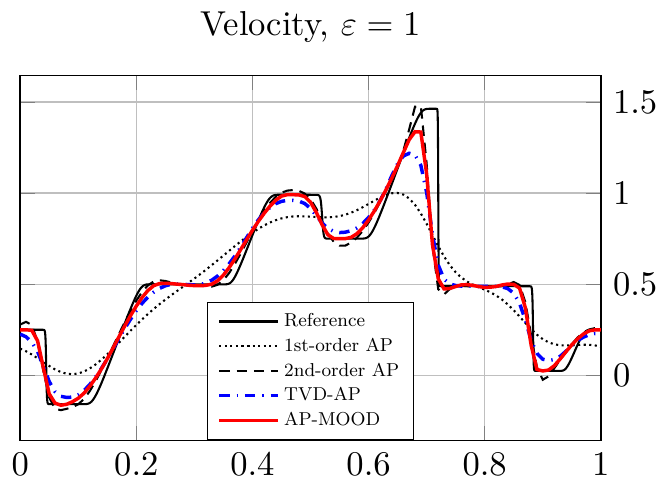}
    \caption{Degond-Tang experiment with $\varepsilon = 1$ and $100$ discretization cells at time $t = 0.075$ and $N=100$ mesh points. Density left image, momentum right image.}
    \label{fig:Degond_Tang_eps_1}
\end{figure}

As in the previous case, the first-order AP scheme is very diffusive and it smears out all shock waves.
The second-order AP scheme yields a less diffusive approximation, but it is not TVD because of overshoots and oscillations, while the \tap and the \tapm scheme decrease the diffusion, and therefore greatly improve the numerical approximation compared to the first-order AP scheme.

\begin{figure}[!ht]
    \centering
    \includegraphics[height=0.3\textwidth]{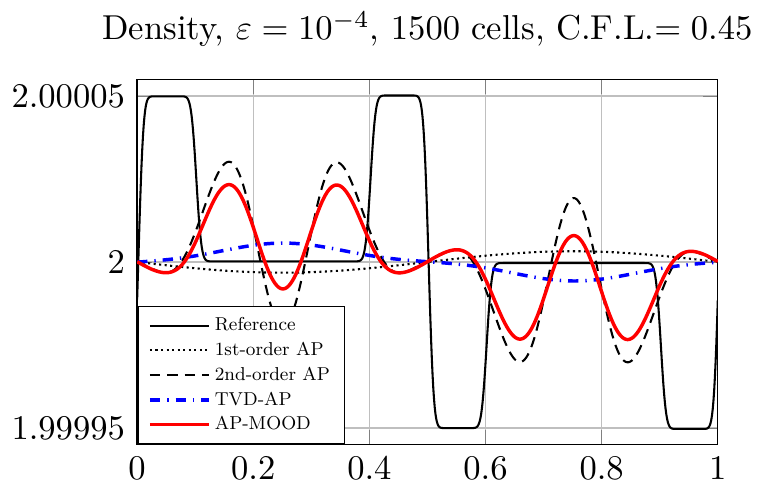} \qquad%
    \includegraphics[height=0.3\textwidth]{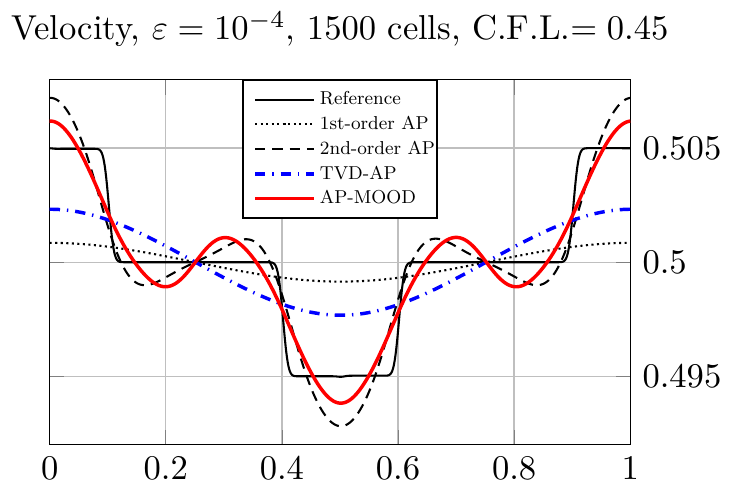}
    \medskip
    \includegraphics[height=0.3\textwidth]{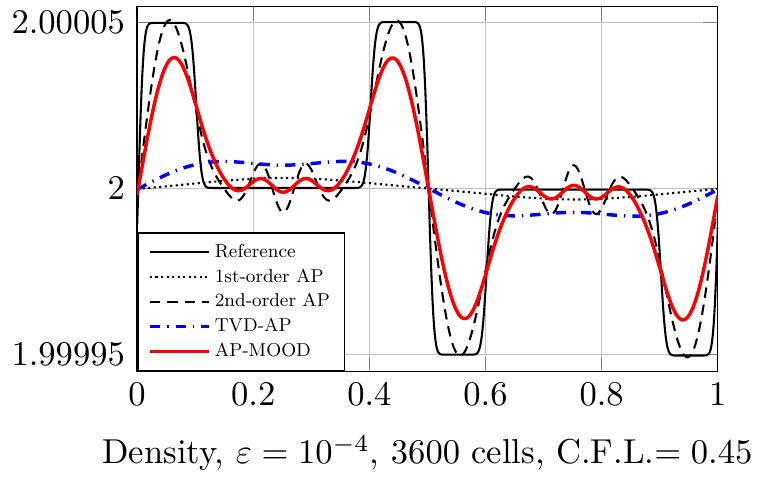} \qquad%
    \includegraphics[height=0.3\textwidth]{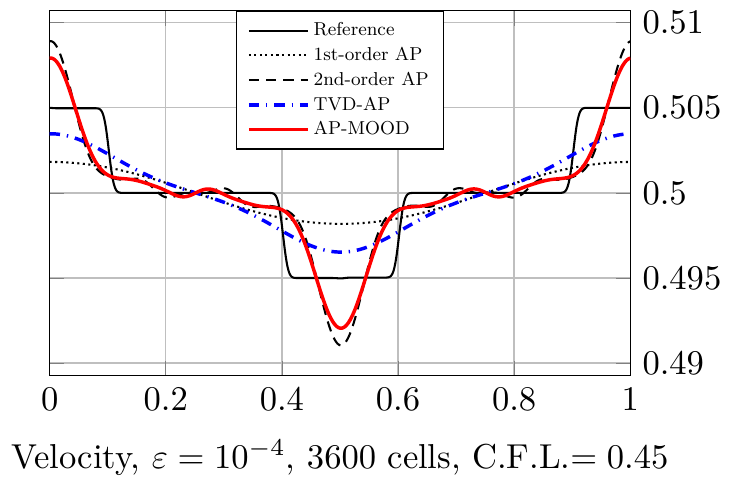}
    \caption{Degond-Tang experiment with $\varepsilon = 10^{-4}$ at time $t = 0.0015$ and $N=1500$ points for the top panels and $N=3600$ for the bottom panels.  Density left images, momentum right images.}
    \label{fig:Degond_Tang_eps_1em4}

\end{figure}
In Figure~\ref{fig:Degond_Tang_eps_1em4}, we observe that while the first-order AP scheme projects the approximate solution onto the incompressible limit
and avoids computing the small structures and the fast waves present in the reference solution, the second order and the \tapm scheme
appropriately capture the micro-structure of the solution, still allowing for much larger time steps.
Therefore, if one is interested into the small structures close to the incompressible limit,
then high accurate numerical schemes seem to be highly relevant.

\subsection{Order of accuracy in two dimensions}
We measure the order of accuracy of the scheme in two space dimensions. We consider $\gamma=1$ and the following smooth exact solution of the isentropic Euler system (\ref{eq:isentropic_2D}):
\begin{equation*}
    \label{eq:smooth_exact_solution_with_constants}
    \begin{aligned}
        \rho_{ex}(x,y,t) &= \rho_\infty - \frac { a^2 \varepsilon } { 8 d } e^{ 2 d ( b - \bar r(x,y,t)^2 ) }, \\
           u_{ex}(x,y,t) &= u_\infty + a \bar y(t) \sqrt{ \frac \gamma 2 } e^{ d ( b - \bar r(x,y,t)^2 ) } \lp \rho_{ex}(x,y,t) \rp^{\frac \gamma 2 - 1 }, \\
           v_{ex}(x,y,t) &= v_\infty - a \bar x(t) \sqrt{ \frac \gamma 2 } e^{ d ( b - \bar r(x,y,t)^2 ) } \lp \rho_{ex}(x,y,t) \rp^{\frac \gamma 2 - 1 },
    \end{aligned}
\end{equation*}
where we have set $\bar r(x,y,t)^2 = \bar x(t)^2 + \bar y(t)^2$, with
$ \bar x(t) = x - x_0 - u_\infty t$,
       $ \bar y(t) = y - y_0 - v_\infty t$.
This exact solution corresponds to a vortex initially centered at $^t(x_0, y_0)$ and moving with the phase velocity~$^t(u_\infty, v_\infty)$.
For the numerical simulation, we take $\rho_\infty = 1$, $a = 1$, $b = 0$, $d = 2$, $x_0 = 0$, $y_0 = 0$ and $^t(u_\infty, v_\infty) = ^t(1, 0)$.
The space domain is $[-1.5,2.5] \times [-2,2]$. The simulations are carried out for three different values of the squared Mach number $\varepsilon$ ($\varepsilon=1$, $\varepsilon=10^{-2}$, $\varepsilon=10^{-4}$) until the final physical time $t_{end} = 1$. To assess the numerical order of accuracy, we compute the following $L^\infty$ errors for several uniform meshes containing:
$$e_\infty^n(\rho) = \max_{j,k} \left| \rho_{j,k}^n - \lp \rho_{ex} \rp_{j,k}^n \right|,
\quad\hbox{ and } \quad e_\infty^n(\rho U) = \max_{j,k} \left| \lp \rho \sqrt{u^2+v^2} \rp_{j,k}^n \right.\left.- \lp \rho_{ex} \sqrt{ u_{ex}^2 + v_{ex}^2 } \rp_{j,k}^n \right|.$$

For the four schemes, the errors are collected in Table~\ref{tab:order_2D_eps_1}  for $\varepsilon = 1$,
                                            Table~\ref{tab:order_2D_eps_1em2}  for $\varepsilon = 10^{-2}$ and
                                            Table~\ref{tab:order_2D_eps_1em4}  for $\varepsilon = 10^{-4}$.
In addition, we display error lines in Figure~\ref{fig:error_lines_2D}.

\begin{table}[!ht]
    \centering
    \begin{tabular}{@{}lcccccccc@{}}
        \toprule
              & \multicolumn{2}{c}{1st-order AP} & \multicolumn{2}{c}{\tap} & \multicolumn{2}{c}{2nd-order AP} & \multicolumn{2}{c}{\tapm} \\
                \cmidrule(lr){2-3}                \cmidrule(lr){4-5}         \cmidrule(lr){6-7}                 \cmidrule(lr){8-9}
        N     & $e_\infty^n(\rho)$ & order & $e_\infty^n(\rho)$ & order & $e_\infty^n(\rho)$ & order & $e_\infty^n(\rho)$ & order \\ \midrule
        625   & 4.30e-02         &  ---  & 1.93e-02         &  ---  & 8.84e-03         &  ---  & 1.04e-02         &  ---  \\
        2500  & 3.36e-02         &  0.35 & 6.05e-03         &  1.67 & 1.66e-03         &  2.41 & 2.14e-03         &  2.28 \\
        10000 & 2.20e-02         &  0.61 & 2.08e-03         &  1.54 & 2.87e-04         &  2.53 & 6.31e-04         &  1.76 \\
        40000 & 1.30e-02         &  0.76 & 7.63e-04         &  1.45 & 5.63e-05         &  2.35 & 1.80e-04         &  1.81 \\
                        \cmidrule(lr){2-3}                \cmidrule(lr){4-5}         \cmidrule(lr){6-7}                 \cmidrule(lr){8-9}
        N     & $e_\infty^n(\rho U)$ & order & $e_\infty^n(\rho U)$ & order & $e_\infty^n(\rho U)$ & order & $e_\infty^n(\rho U)$ & order \\ \midrule
        625   & 1.07e-01         &  ---  & 4.61e-02         &  ---  & 1.62e-02         &  ---  & 2.26e-02         &  ---  \\
        2500  & 7.59e-02         &  0.50 & 1.25e-02         &  1.88 & 3.02e-03         &  2.42 & 4.40e-03         &  2.36 \\
        10000 & 4.73e-02         &  0.68 & 5.19e-03         &  1.27 & 5.33e-04         &  2.50 & 1.47e-03         &  1.59 \\
        40000 & 2.69e-02         &  0.81 & 2.54e-03         &  1.03 & 1.09e-04         &  2.29 & 4.84e-04         &  1.60 \\ \bottomrule
    \end{tabular}
    \caption{Density and momentum norm  errors and order of accuracy with $\varepsilon = 1$ for the four schemes.}
    \label{tab:order_2D_eps_1}
\end{table}

\begin{table}[!ht]
    \centering
    \begin{tabular}{@{}lcccccccc@{}}
        \toprule
              & \multicolumn{2}{c}{1st-order AP} & \multicolumn{2}{c}{\tap} & \multicolumn{2}{c}{2nd-order AP} & \multicolumn{2}{c}{\tapm} \\
                \cmidrule(lr){2-3}                \cmidrule(lr){4-5}         \cmidrule(lr){6-7}                 \cmidrule(lr){8-9}
        N     & $e_\infty^n(\rho)$ & order & $e_\infty^n(\rho)$ & order & $e_\infty^n(\rho)$ & order & $e_\infty^n(\rho)$ & order \\ \midrule
        625   & 5.58e-04         &  ---  & 3.57e-04         &  ---  & 1.57e-04         &  ---  & 2.46e-04         &  ---  \\
        2500  & 5.16e-04         &  0.11 & 1.41e-04         &  1.34 & 3.31e-05         &  2.25 & 4.49e-05         &  2.46 \\
        10000 & 4.20e-04         &  0.30 & 4.94e-05         &  1.52 & 4.68e-06         &  2.82 & 1.68e-05         &  1.42 \\
        40000 & 3.02e-04         &  0.48 & 1.55e-05         &  1.67 & 6.33e-07         &  2.89 & 4.37e-06         &  1.94 \\
                \cmidrule(lr){2-3}                \cmidrule(lr){4-5}         \cmidrule(lr){6-7}                 \cmidrule(lr){8-9}
        N     & $e_\infty^n(\rho U)$ & order & $e_\infty^n(\rho U)$ & order & $e_\infty^n(\rho U)$ & order & $e_\infty^n(\rho U)$ & order \\ \midrule
        625   & 1.51e-01         &  ---  & 7.79e-02         &  ---  & 3.19e-02         &  ---  & 3.88e-02         &  ---  \\
        2500  & 1.28e-01         &  0.24 & 2.84e-02         &  1.46 & 6.04e-03         &  2.40 & 6.81e-03         &  2.51 \\
        10000 & 9.52e-02         &  0.43 & 9.35e-03         &  1.60 & 8.50e-04         &  2.83 & 1.38e-03         &  2.30 \\
        40000 & 6.31e-02         &  0.59 & 2.81e-03         &  1.74 & 1.15e-04         &  2.89 & 4.57e-04         &  1.59 \\ \bottomrule
    \end{tabular}
    \caption{Density and momentum norm errors and order of accuracy with $\varepsilon = 10^{-2}$ for the four schemes.}
    \label{tab:order_2D_eps_1em2}
\end{table}

\begin{table}[!ht]
    \centering
    \begin{tabular}{@{}lcccccccc@{}}
        \toprule
              & \multicolumn{2}{c}{1st-order AP} & \multicolumn{2}{c}{\tap} & \multicolumn{2}{c}{2nd-order AP} & \multicolumn{2}{c}{\tapm} \\
                \cmidrule(lr){2-3}                \cmidrule(lr){4-5}         \cmidrule(lr){6-7}                 \cmidrule(lr){8-9}
        N      & $e_\infty^n(\rho)$ & order & $e_\infty^n(\rho)$ & order & $e_\infty^n(\rho)$ & order & $e_\infty^n(\rho)$ & order \\ \midrule
           625 & 2.42e-05         &  ---  & 1.12e-05         &  ---  & 5.32e-06         &  ---  & 6.33e-06         &  ---  \\
          2500 & 2.21e-05         &  0.13 & 1.27e-05         & -0.18 & 1.75e-06         &  1.60 & 1.79e-06         &  1.82 \\
         10000 & 1.17e-05         &  0.91 & 2.97e-06         &  2.10 & 8.31e-07         &  1.08 & 7.88e-07         &  1.19 \\
         40000 & 9.33e-06         &  0.33 & 2.06e-06         &  0.53 & 1.19e-07         &  2.80 & 1.18e-07         &  2.74 \\
                         \cmidrule(lr){2-3}                \cmidrule(lr){4-5}         \cmidrule(lr){6-7}                 \cmidrule(lr){8-9}
        N      & $e_\infty^n(\rho U)$ & order & $e_\infty^n(\rho U)$ & order & $e_\infty^n(\rho U)$ & order & $e_\infty^n(\rho U)$ & order \\ \midrule
           625 & 1.61e-01         &  ---  & 8.81e-02         &  ---  & 3.74e-02         &  ---  & 4.43e-02         &  ---  \\
          2500 & 1.43e-01         &  0.16 & 4.40e-02         &  1.00 & 8.76e-03         &  2.09 & 9.17e-03         &  2.27 \\
         10000 & 1.17e-01         &  0.29 & 1.72e-02         &  1.36 & 1.65e-03         &  2.41 & 1.75e-03         &  2.39 \\
         40000 & 8.59e-02         &  0.45 & 5.69e-03         &  1.59 & 3.06e-04         &  2.43 & 3.05e-04         &  2.52 \\ \bottomrule
    \end{tabular}
    \caption{Density and momentum norm errors and order of accuracy with $\varepsilon = 10^{-4}$ for the four schemes.}
    \label{tab:order_2D_eps_1em4}
\end{table}

\begin{figure}[!ht]
    \centering
    \includegraphics[width=0.4\textwidth]{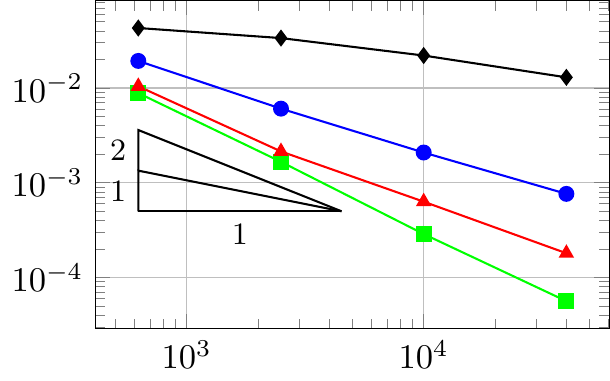} \qquad%
    \includegraphics[width=0.4\textwidth]{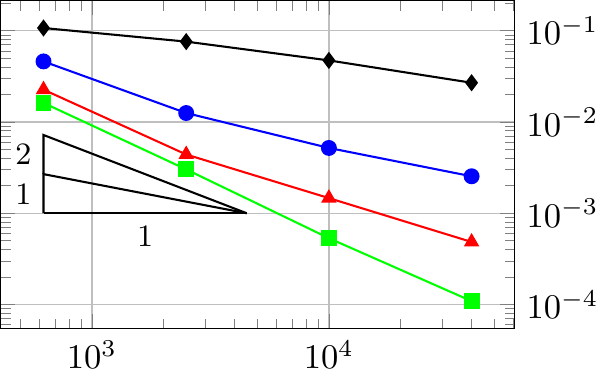}
    \medskip
    \includegraphics[width=0.4\textwidth]{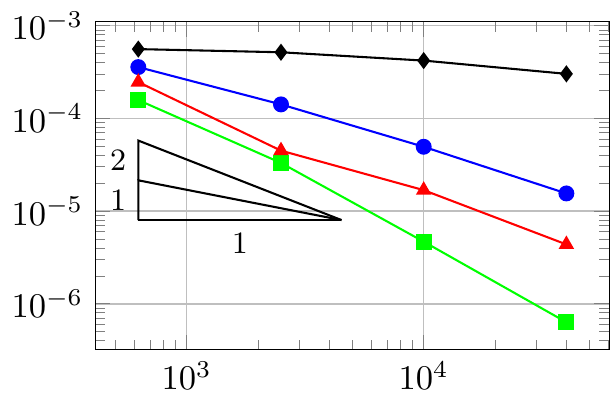} \qquad%
    \includegraphics[width=0.4\textwidth]{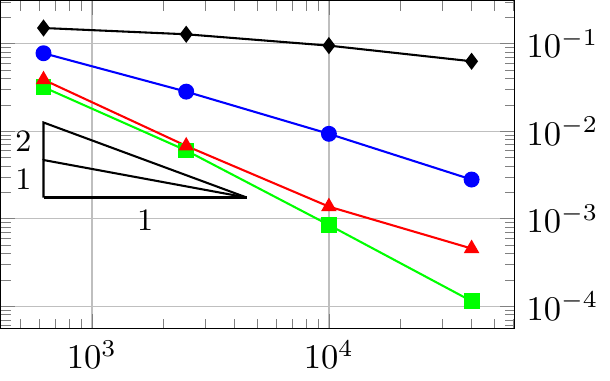}
    \medskip
    \includegraphics[width=0.4\textwidth]{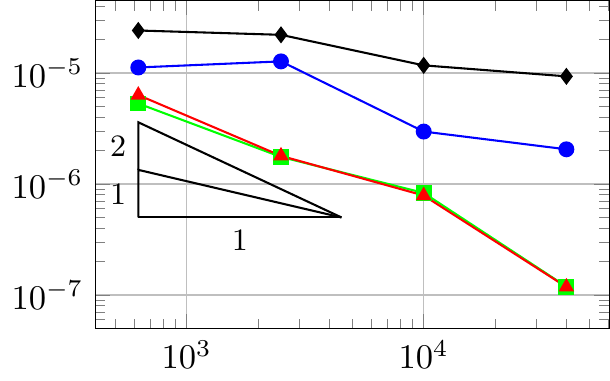} \qquad%
    \includegraphics[width=0.4\textwidth]{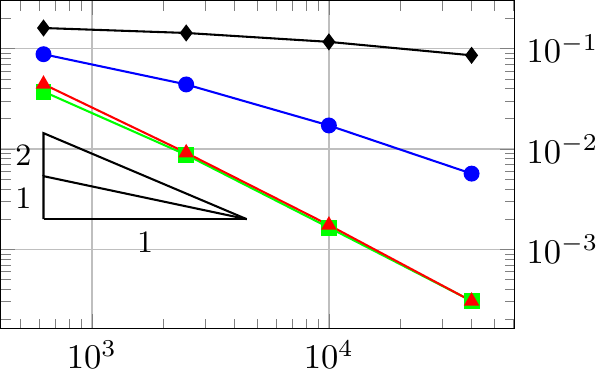}
    \caption{Error lines for the 2D steady vortex.
             From top to bottom: $\varepsilon = 1$, $\varepsilon = 10^{-2}$ and $\varepsilon = 10^{-4}$.
             Left panels: density errors; right panels: momentum errors. Black line corresponds to the first order AP scheme, blue line to the \tap scheme,
             red line to the \tapm scheme, yellow line to the second order AP unlimited.}
    \label{fig:error_lines_2D}

\end{figure}
From these results, we draw similar conclusions as in the 1D case. The errors of \tap and the \tapm schemes lie in between the first and the second order slopes as well as the errors.

\subsection{Asymptotically consistency of the schemes in two dimensions}
We now perform a 2D validation experiment, initially described in \cite{EShu} and used more recently in \cite{BosRusSca}.
It is particularly relevant since, for small values of $\eps$ we can compare the compressible numerical approximations to an approximate solution of the incompressible
Euler equations. In this way, we can measure the asymptotically consistency of our approximations in the low Mach number limit.

The initial data are well-prepared.
Indeed, on the space domain $[0, 2\pi]^2$, we take a constant density $\rho = \pi / 15$, and the initial incompressible velocity field $U=(u,v)$ is given by:
\begin{equation*}
    \label{eq:initial_condition_double_shear_layer}
    u(x,y,0) =
    \begin{cases}
         \tanh ( ( y - \pi / 2 ) / \rho ), & \text{if } y \leq \pi, \\
         \tanh ( ( 3 \pi / 2 - y ) / \rho ), & \text{otherwise,}
    \end{cases}
    \qquad
    v(x,y,0) = 0.05 \sin(x).
\end{equation*}
In addition, we take $\gamma = 1$ and
we prescribe periodic boundary conditions for the compressible Euler system.

To determine the incompressible approximate solution, we consider the vorticity formulation of the incompressible Euler system, given by:
\begin{equation}
    \label{eq:Euler_vorticity}
    \pt \omega + U \cdot \nabla \omega = 0,
\end{equation}
and we recall that the vorticity $\omega$ is given by $\omega = \partial_x v - \partial_y u$.
Since $\nabla \cdot U = 0$, there exists a stream function $\psi$ such that $U = {^t(} \partial_y \psi, -\partial_x \psi)$ and $- \Delta \psi = \omega$.
From these observations, we can obtain a reference solution.
We compute the time evolution of the vorticity by repeating the following three steps:
we first compute the stream function $\psi$, then the associated velocity field, and finally the time update of the vorticity with \eqref{eq:Euler_vorticity}.
To solve the Poisson equation $- \Delta \psi = \omega$, we use a classical discretization of the Laplace operator, and we prescribe periodic boundary conditions.
Since this leads to a singular system, we also impose that the stream function has a null average.
This does not alter the rest of the procedure since we are only interested in the derivatives of $\psi$.
The velocity is then obtained by an application of a centered gradient discretization, and an upwind finite difference scheme provides an approximate solution for \eqref{eq:Euler_vorticity}.
Periodic boundary conditions are prescribed in both of these steps.

We stress that the reference solution is obtained from the incompressible Euler equations while
the schemes under consideration approximately solve the compressible Euler system with a very small Mach number.
The results are given in Figure~\ref{fig:double_shear_layer_eps_1em5}. We take $\varepsilon = 10^{-5}$ and $t_{end} = 6$ to compare the numerical solutions provided
by the four schemes for the compressible Euler system with the reference incompressible solution.
The mesh is constituted of $200\times 200$ cells.

We can see that the first-order scheme loses the main structure of the solution (it is worth noting that, on a finer grid, the structure
can be captured
by the first-order scheme).
The limited \tap scheme provides a smeared numerical approximation, while the \tap and second-order schemes yield similar numerical solutions.
We note that, for these
schemes, the main structure of the solution is captured.
However, the small central structures are smeared
because the grid is too coarse.
Overall, the proposed compressible schemes offer a convincing approximation of the incompressible solution when $\varepsilon$ is small enough.

\begin{figure}[!ht]

    \centering

    \begin{minipage}{0.3\textwidth}
        \includegraphics[width=\textwidth]{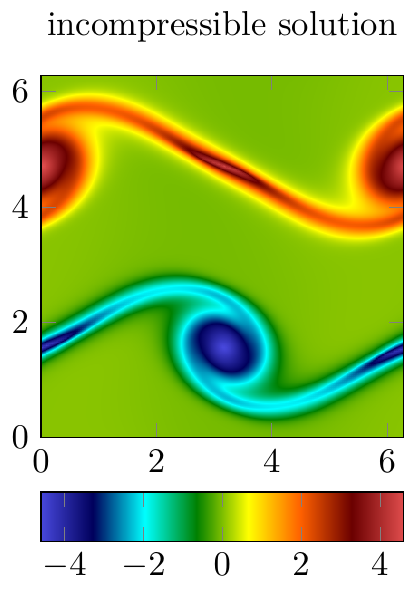}
    \end{minipage}%
    \begin{minipage}{0.3\textwidth}
        \vspace{-31pt}
        \includegraphics[width=\textwidth]{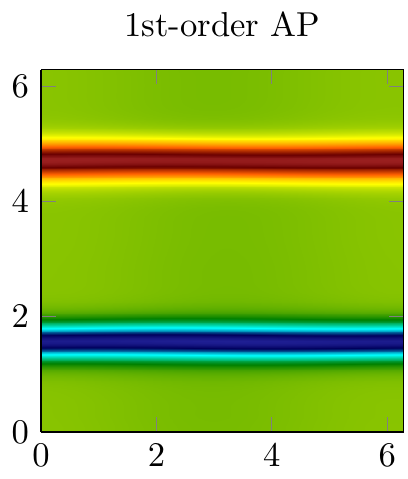}
    \end{minipage}%
    \begin{minipage}{0.3\textwidth}
        \vspace{-31pt}
        \includegraphics[width=\textwidth]{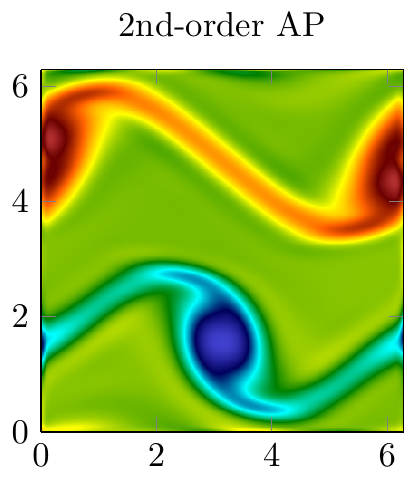}
    \end{minipage}

    \makebox[0.3\textwidth]{}%
    \begin{minipage}{0.3\textwidth}
        \includegraphics[width=\textwidth]{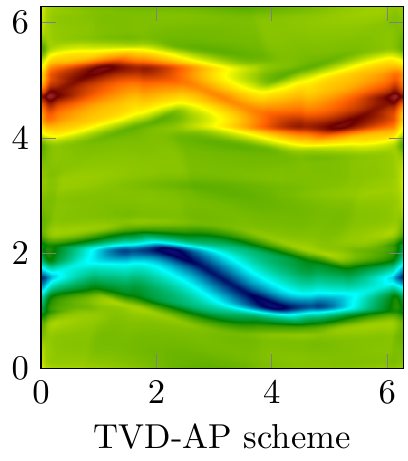}
    \end{minipage}%
    \begin{minipage}{0.3\textwidth}
        \includegraphics[width=\textwidth]{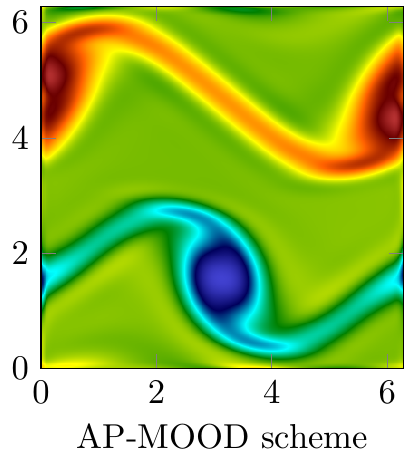}
    \end{minipage}

    \caption{Numerical solution for the double shear layer experiment with $\varepsilon = 10^{-5}$, using $200 \times 200 = 40000$ cells.}
    \label{fig:double_shear_layer_eps_1em5}

\end{figure}

Now, we use this test case in the $\varepsilon = 1$.  The numerical solution given by the four schemes at the final time $t_{end} = 10$ is compared in
Figure~\ref{fig:double_shear_layer_eps_1} to a reference solution given by the first-order AP scheme with a very fine
mesh.
The reference solution is obtained by using the first-order scheme with $400 \times 400$ cells.
In this figure, we
represent the vorticity of the solution, given by $\omega = \partial_x v - \partial_y u$, since this quantity is relevant for small $\varepsilon$.\\
In Figure~\ref{fig:double_shear_layer_eps_1}
are plotted the results when a coarse mesh made of $40\times 40$ cells is employed.
We note that the main structures of the reference solution are captured by the second-order and the \tapm scheme, with the \tap scheme being
slightly
more diffusive.
However, the first-order scheme is so diffusive that it destroys most of the structures.
In the bottom left corner of Figure~\ref{fig:double_shear_layer_eps_1}, we have added a zoom on the domain $[2.35,3.85] \times [4.4,5]$ of the reference and the second-order solutions.
We note that the very fine structure present in this domain is smeared by the second-order scheme due to the use of too coarse a mesh.

\begin{figure}[!ht]

    \centering

    \begin{minipage}{0.3\textwidth}
        \includegraphics[width=\textwidth]{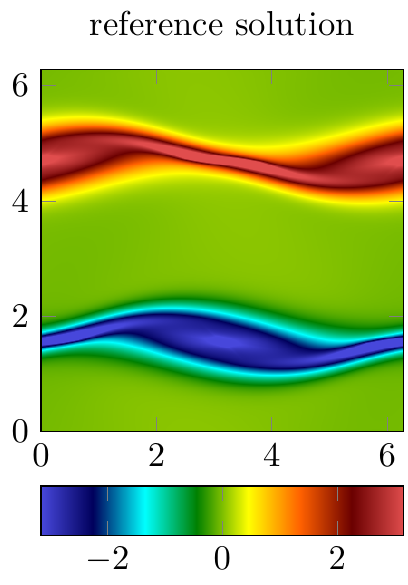}
    \end{minipage}%
    \begin{minipage}{0.3\textwidth}
        \vspace{-33pt}
        \includegraphics[width=\textwidth]{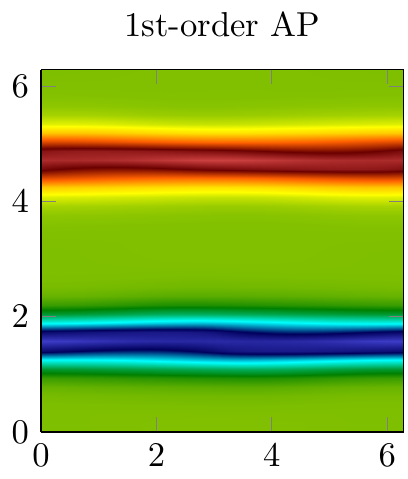}
    \end{minipage}%
    \begin{minipage}{0.3\textwidth}
        \vspace{-33pt}
        \includegraphics[width=\textwidth]{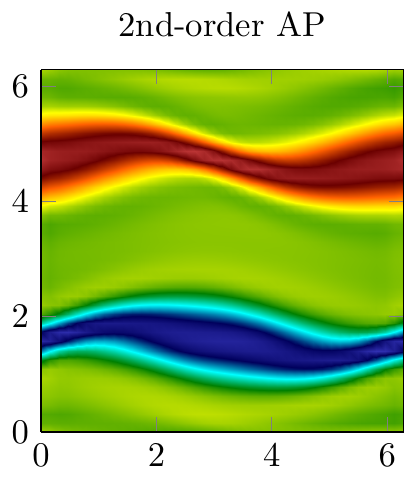}
    \end{minipage}

    \begin{minipage}{0.3\textwidth}
        \includegraphics[width=\textwidth]{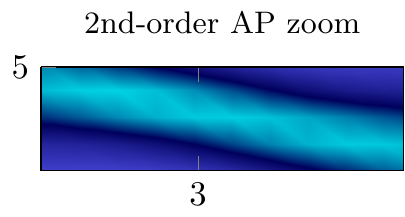}
        \includegraphics[width=\textwidth]{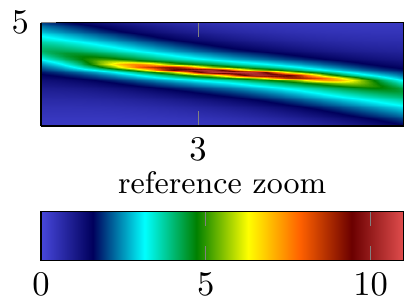}
    \end{minipage}%
    \begin{minipage}{0.3\textwidth}
        \includegraphics[width=\textwidth]{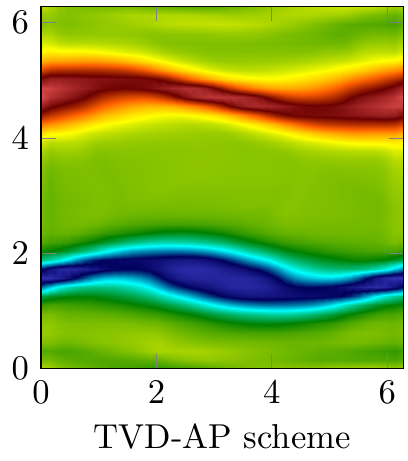}
    \end{minipage}%
    \begin{minipage}{0.3\textwidth}
        \includegraphics[width=\textwidth]{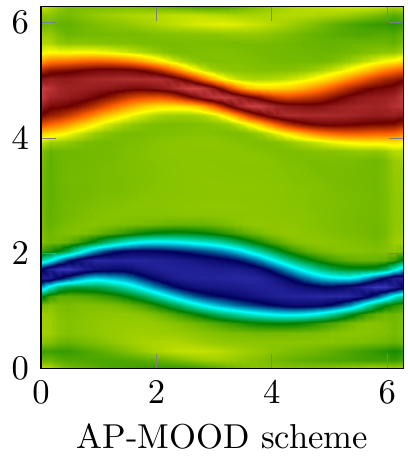}
    \end{minipage}

    \caption{Numerical solution for the double shear layer experiment with $\varepsilon = 1$.
             The numerical solutions are made with $40 \times 40 = 1600$ cells.
             The bottom left of this figure contains a zoom on a very small structure smeared by the coarse mesh but present in the reference solution.
             The color legend in the bottom left corner is only applied to this zoom.}
    \label{fig:double_shear_layer_eps_1}

\end{figure}


\section{Conclusion}
In this paper we have derived a new second order scheme for the compressible Euler equations in the low Mach number regime.
Since, non physical oscillations cannot be avoided for C.F.L. conditions larger than the one imposed by explicit methods, we have constructed a new method based on the
coupling of first order with second order in time and space schemes. This approach has permitted to get an highly accurate and asymptotic preserving scheme which additionally enjoys the TVD and $L^\infty$ property independently from the time step. Successively, the introduction of limiters has permitted a passage from the second to the TVD method only when strictly necessary further improving the overall accuracy. For all the schemes presented, the stability constraints did not depend on the Mach number value and these schemes degenerate into a consistent highly accurate discretization of the incompressible system in the low Mach limit.
Numerical experiments supported the proposed analysis. In the future, we aim in focusing on the generalization of such technique to the case of TVD schemes which couple AP schemes of order higher than two with first order in time AP methods. Moreover, we aim in exploring more in depth the use of limiters since in some cases some small oscillations remain present for the limited method. Local coupling techniques between the different order schemes can also largely improve the results obtained and they are now the subject of investigations. Extension to the full Euler equations are under study.


\end{document}